\documentclass[12pt]{amsart}
\usepackage{graphicx}
\usepackage[active]{srcltx}
\usepackage{amsthm,mathrsfs}
\usepackage{a4wide}
\usepackage[english]{babel}
\usepackage{amsfonts}
\usepackage{amsmath}
\usepackage{amssymb}
\newtheorem{lemma}{Lemma}
\newtheorem{theorem}{Theorem}

\usepackage{float}
\usepackage{url}
\usepackage{enumitem}

\newcommand {\p} {\mathbb{P}}

\newcommand {\R} {\mathbb{R}}

\newcommand {\ve} {\varepsilon}

\makeatletter\def\blfootnote{\xdef\@thefnmark{}\@footnotetext}\makeatother
\allowdisplaybreaks
\title[On parametric Thue--Morse Sequences and Lacunary Trigonometric Products]{\bf On parametric Thue--Morse Sequences and Lacunary Trigonometric Products}

\author{Christoph Aistleitner} 
\address{Institute of Financial Mathematics and Applied Number Theory, Johannes Kepler University Linz}
\email{aistleitner@math.tugraz.at}

\author{Roswitha Hofer} 
\address{Institute of Financial Mathematics and Applied Number Theory, Johannes Kepler University Linz}
\email{roswitha.hofer@jku.at}

\author{Gerhard Larcher} 
\address{Institute of Financial Mathematics and Applied Number Theory, Johannes Kepler University Linz}
\email{gerhard.larcher@jku.at}

\thanks{The first author is supported by a Schr\"odinger scholarship of the Austrian
Research Foundation (FWF). The second and the third author are supported by the Austrian Science Fund (FWF): Project F5505-N26 and F5507-N26, which are part of the Special Research Program ``Quasi-Monte Carlo Methods: Theory and Applications''}


\begin{document}

\begin{abstract}
One of the fundamental theorems of uniform distribution theory states that the fractional parts of the sequence $(n \alpha)_{n \geq 1}$ are uniformly distributed modulo one (u.d. mod 1) for every irrational number $\alpha$. Another important result of Weyl states that for every sequence $(n_k)_{k \geq 1}$ of distinct positive integers the sequence of fractional parts of $(n_k \alpha)_{k \geq 1}$ is u.d. mod 1 for almost all $\alpha$. However, in this general case it is usually extremely difficult to classify those $\alpha$ for which uniform distribution occurs, and to measure the speed of convergence of the empirical distribution of $(\{n_1 \alpha\}, \dots, \{n_N \alpha\})$ towards the uniform distribution. In the present paper we investigate this problem in the case when $(n_k)_{k \geq 1}$ is the Thue--Morse sequence of integers, which means the sequence of positive integers having an even sum of digits in base 2. In particular we utilize a connection with lacunary trigonometric products $\prod^{L}_{\ell=0} 
\left|\sin \pi 2^{\ell} \alpha \right|$, and by giving sharp metric estimates for such products we derive sharp metric estimates for exponential sums of $\left(n_{k} \alpha\right)_{k \geq 1}$ and for the discrepancy of $\left(\left\{n_{k} \alpha\right\}\right)_{k \geq 1}.$ Furthermore, we comment on the connection between our results and an open problem in the metric theory of Diophantine approximation, and we provide some explicit examples of numbers $\alpha$ for which we can give estimates for the discrepancy of $\left(\left\{n_{k} \alpha\right\}\right)_{k \geq1}$.
\end{abstract}

\date{}
\maketitle

\section{Introduction and statement of results} \label{sect_int}

Throughout the rest of this paper, let $(n_k)_{k \geq 1}$ denote the sequence of positive integers which have an even sum-of-digits function in base $2$, sorted in increasing order. In other words, $(n_k)_{k \geq 1}$ is the sequence of Thue--Morse integers $(0,3,5,6,9,10,12,\dots)$. Furthermore, we write $(m_k)_{k \geq 1}$ for the sequences of those numbers which are \emph{not} contained in $(n_k)_{k \geq 1}$, sorted in increasing order; thus $(m_k)_{k \geq 1} = (1,2,4,7,8,11,\dots)$. The numbers $(n_k)_{k \geq 1}$ are frequently called \emph{evil numbers}, while the numbers $(m_k)_{k \geq 1}$ are called \emph{odious numbers}.\\

The Thue--Morse integers are characterized by the Thue--Morse sequence 
$$
\left(t_{n}\right)_{n \geq 0} = \left(0,1,1,0,1,0,0,1,1,0,0,1,0,1,1,0,\ldots\right)
$$
which has been discovered several times in the literature. For an extensive survey, see~\cite{Shallit}. In our notation we have $n \in (n_k)_{k\geq1}$ if and only if $t_{n}=0.$\\

In this paper we analyze exponential sums of the form $\sum^{N}_{k=1} e^{2 \pi i n_{k} \alpha}$ for reals $\alpha \in \left[\left.0,1\right)\right.$, and -- what is intimately connected -- products of the form $\prod^{L}_{\ell=0} \left|\sin \pi 2^{\ell} \alpha\right|$, as well as distribution properties of the sequence $\left(\left\{n_{k} \alpha\right\}\right)_{k \geq 1}$. To quantify the regularity of the distribution of a finite set of real numbers in $[0,1]$ we use the notion of the \emph{star-discrepancy} $D_N^*$. For given numbers $x_1, \dots, x_N$, their star-discrepancy is defined by
$$
D_N^* (x_1, \dots, x_N) = \sup_{a \in [0,1]} \left| \frac{1}{N} \sum_{n=1}^N \mathbf{1}_{[0,a]} (x_n) - a \right|.
$$
An infinite sequence $(x_n)_{n \geq 1}$ whose discrepancy $D_N^*$ tends to zero as $N \to \infty$ is called \emph{uniformly distributed modulo one} (u.d. mod 1). Informally speaking, the star-discrepancy is a measure for the deviation between uniform distribution on $[0,1]$ and the empirical distribution of a given point set; in probabilistic terminology this corresponds to the Kolmogorov--Smirnov statistic. Discrepancy theory is a rich subject, which has close links to number theory, probability theory, ergodic theory and numerical analysis. For more information on discrepancy theory, we refer to the standard monographs~\cite{dts,knu}.\\

Sequences of the form $\left(\left\{n \alpha\right\}\right)_{n \geq 1}$ are called \emph{Kronecker sequences}. One of the fundamental results of discrepancy theory states that such a sequence is u.d. mod 1 if and only if $\alpha$ is irrational. It is also well-known that the discrepancy of such a sequence depends on Diophantine approximation properties of $\alpha$. More precisely, we have
$$
\Omega \left(\sum^{m(N)}_{n=1} a_{n}\right) = N D^{*}_{N} (\{\alpha\}, \dots, \{N \alpha\}) = \mathcal{O}\left(\sum^{m(N)}_{n=1} a_{n}\right)
$$
as $N \to \infty$, where $a_{1}, a_{2}, a_{3}, \ldots$ are the continued fraction coefficients of $\alpha$ and where $m(N)$ is defined by $q_{m(N)-1} < N \leq q_{m(N)}$ with $q_{1}<q_{2}<q_{3}< \ldots$ denoting the best approximation denominators of $\alpha$ (see for example~\cite[Corollary~1.64]{dts}). Hence 
$$
N D^{*}_{N} (\{\alpha\}, \dots, \{N \alpha\}) = \mathcal{O} \left(\log N\right) \qquad \textrm{as $N \to \infty$}
$$ 
if $\alpha$ has bounded continued fraction coefficients, and, as a consequence of metric results of Khintchine~\cite{khin}, for every $\ve > 0$ we have
$$
N D^{*}_{N} (\{\alpha\}, \dots, \{N \alpha\}) = \mathcal{O} \left( \frac{(\log N) (\log \log N)^{1+\ve})}{N} \right) \qquad \textrm{as $N \to \infty$}
$$ 
for almost all $\alpha \in \mathbb{R}$.\\\\

It is known (and it will be re-proved implicitly in this paper; see Section 6) that the sequence $\left(\left\{n_{k} \alpha\right\}\right)_{k\geq 1}$ (which we will call \emph{Thue--Morse--Kronecker sequence}) is also uniformly distributed in the unit interval if and only if $\alpha$ is irrational. However, it turns out to be a very difficult task to give sharp estimates for the discrepancy of this sequence for concrete values of $\alpha$. As we will see, the discrepancy of a Thue--Morse--Kronecker sequence $\left(\left\{n_{k} \alpha \right\}\right)_{k \geq 1}$ depends on {D}iophantine approximation properties {\em and} properties of the digit representation of $\alpha$ in base 2. Until now there are only few (non-trivial) cases of $\alpha$ where we have enough information about both of these aspects.\\

Exponential sums and discrepancy theory are intimately connected. One such connection is Weyl's criterion, two others are the Erd\H os--Tur\'an inequality and Koksma's inequality. The Erd\H os--Tur\'an inequality (see for example~\cite{dts,knu,rt}) states that for points $x_1, \dots, x_N \in [0,1]$ we have
\begin{eqnarray} \label{erdtur}
D_N^* (x_1, \dots, x_N) \leq \frac{1}{H+1} + \sum_{h=1}^H \frac{1}{h} \left|\frac{1}{N} \sum_{k=1}^N e^{2 \pi i h x_k} \right|,
\end{eqnarray}
where $H$ is an arbitrary positive integer. Koksma's inequality says that 
\begin{equation} \label{koksma}
\left| \int_0^1 f(x)~dx - \frac{1}{N}\sum_{k=1}^N f(x_k) \right| \leq (\textup{Var}_{[0,1]} f) D_N^*(x_1, \dots, x_N),
\end{equation}
for any function $f$ having bounded variation on $[0,1]$. When combined, the Erd\H os--Tur\'an inequality and Koksma's inequality show that exponential sums can be used to obtain both upper and lower bounds for the discrepancy.\\

As an explicit lower bound from~\eqref{koksma} we get (compare for example~\cite{knu}):
\begin{equation} \label{BL1}
D^{*}_{N}  \left(x_{1}, \ldots, x_{N}\right) \geq \frac{1}{4H} \left|\frac{1}{N}\sum^{N}_{k=1} e^{2 \pi i H x_{k}}\right|,
\end{equation}
where $H$ is an arbitrary positive integer. Koksma's inequality and its multi-dimensional generalization are also the cornerstone of the application of low-discrepancy point sets in numerical integration (so-called \emph{Quasi-Monte Carlo integration}; see for example~\cite{dpd,nwt}).\\

Consequently, in this paper we will mainly be concerned with the problem of investigating exponential sums of the form $\sum^{N}_{k=1} e^{2 \pi i n_{k} \alpha}$. It turns out that this investigation relies on studying lacunary products of the form $\prod^{L}_{\ell=0} \left|\sin \pi 2^{\ell} \alpha\right|$. Furthermore we study the discrepancy of $\left(\left\{n_{k} \alpha\right\}\right)_{k\geq 1}$. For all three topics we obtain sharp metric results. The investigation of the lower bound for the discrepancy leads to a challenging open problem in Diophantine approximation. 
Finally, we consider two concrete non-trivial special examples for $\alpha$.\\

The main results of this paper are the following. (Throughout the rest of this paper, we write $\exp(x)$ for $e^x$.)\\

\begin{theorem} \label{th_thue}
Let $(n_k)_{k \geq 1}$ be the sequence of Thue--Morse integers, and let $h \neq 0$ be an integer. Let $\ve>0$ be arbitrary. Then for almost all $\alpha \in (0,1)$ we have
\begin{equation} \label{ththue1a}
\left|\sum_{k=1}^N e^{2 \pi i h n_k \alpha}\right| ~\leq~ \exp\left(\left(\frac{\pi}{\sqrt{\log 2}}+\ve\right) (\log N)^{1/2} (\log \log \log N)^{1/2}\right)
\end{equation} 
for all $N \geq N_0 (\alpha,h, \ve)$, and 
\begin{equation} \label{ththue1b}
\left|\sum_{k=1}^N e^{2 \pi i h n_k \alpha} \right| ~\geq~ \exp \left(\left(\frac{\pi}{\sqrt{\log 2}}-\ve\right) (\log N)^{1/2} (\log \log \log N)^{1/2}\right)
\end{equation}
for infinitely many $N$.
\end{theorem}
~\\

Note that the exponential function in~\eqref{ththue1a} grows more slowly than any (fixed) power of $N$; but faster than any (fixed) power of $\log N$. In other words, as a consequence of Theorem~\ref{th_thue} for every $\ve>0$ and every $A > 0$ we have
$$
\Omega \left(\left(\log N\right)^{A}\right) = \left|\sum_{k=1}^N e^{2 \pi i h n_k \alpha}\right| = \mathcal{O} \left(N^{\ve}\right) \qquad \textrm{as $N \to \infty$},
$$
for almost all $\alpha$.\\

It will turn out that Theorem~\ref{th_thue} is an almost immediate consequence of the following result on lacunary trigonometric products.\\

\begin{theorem} \label{th_theorem2}
Let $ \ve > 0$ be arbitrary. Then for almost all $\alpha \in \left(0,1\right)$ we have
\begin{equation} \label{FO2}
\prod^{L}_{\ell=0} \left|2 \sin \pi 2^{\ell} \alpha\right| \leq \exp\left(\left(\pi+\ve\right)\sqrt{L \log \log L}\right)
\end{equation}
for all $L \geq L_0\left(\alpha, \ve\right)$, and
\begin{equation} \label{FO3}
\prod^{L}_{\ell=0} \left|2 \sin \pi 2^{\ell} \alpha \right| \geq \exp\left(\left(\pi- \ve\right)\sqrt{L \log \log L}\right)
\end{equation}
for infinitely many $L$.
\end{theorem}
~\\ 

This result is a consequence of a more general result, Theorem~\ref{lac_co1}, which will be formulated later in Section~\ref{sect_lac} since it needs some technical prerequisites.\\

From the lower bound in Theorem~\ref{th_thue} and formula~\eqref{BL1} we immediately obtain a metric lower bound for the discrepancy $D_{N}^{*} \left(\left\{n_{1} \alpha\right\}, \ldots, \left\{n_{N} \alpha\right\}\right)$ of the Thue--Morse--Kronecker sequence. However in Theorem~\ref{th_thue2} it turns out that the true metric order of the discrepancy $D_{N}^{*} \left(\left\{n_{1} \alpha\right\}, \ldots, \left\{n_{N} \alpha\right\}\right)$ is much larger.\\

\begin{theorem} \label{th_thue2}
Let $\left(n_{k}\right)_{k \geq 1}$ be the sequence of Thue--Morse integers. Let $\ve >0$ be arbitrary. Then for almost all $\alpha \in \left(0,1\right)$ we have
\begin{equation} \label{F3}
N D^{*}_{N} \left(\left\{n_{1} \alpha\right\}, \ldots, \left\{n_{N} \alpha\right\}\right)= \mathcal{O} \left(N^{1+\frac{\log \lambda}{\log 2} + \ve}\right) \qquad \mbox{as}~ N \rightarrow \infty,
\end{equation}
and
\begin{equation} \label{F4}
N D^{*}_{N} \left(\left\{n_{1} \alpha\right\}, \ldots, \left\{n_{N} \alpha\right\}\right) \geq \left(N^{1+\frac{\log \lambda}{\log 2} - \ve}\right) \qquad \mbox{as}~ N \rightarrow \infty
\end{equation}
for infinitely many $N$. Here $\lambda$ is a real constant defined below for which it is known that 
\begin{equation} \label{*}
0.66130 < \lambda < 0.66135.
\end{equation}
\end{theorem}
~\\

The number $\lambda$ in Theorem~\ref{th_thue2} appears in a result of Fouvry and Mauduit~\cite{foumau}, which states that
\begin{equation} \label{FO4}
I_1(L) := \int^{1}_{0} \prod^{L-1}_{\ell=0} \left|\sin \pi 2^{\ell} \alpha \right|d \alpha = \kappa \lambda^{L} \left(1+o(1)\right)
\end{equation}
for $L\rightarrow\infty$, with constants $\kappa > 0$ and $\lambda$ with $0.654336 < \lambda <0.663197.$ In Lemma~\ref{lem7}, which is contained in Section~\ref{sect_proof_th2}, we will improve the estimate for $\lambda$ to~\eqref{*}. Note that as a consequence of~\eqref{F3} and~\eqref{*} we have
$$
N D^{*}_{N} \left(\left\{n_{1} \alpha\right\}, \ldots, \left\{n_{N} \alpha\right\}\right) = \mathcal{O} \left( N^{0.404} \right) \qquad \textrm{as $N \to \infty$}, 
$$
for almost all $\alpha$. This should be compared with the general metric discrepancy bound
\begin{equation} \label{baker}
N D^{*}_{N} \left(\left\{b_{1} \alpha\right\}, \ldots, \left\{b_{N} \alpha\right\}\right) = \mathcal{O} \left( \sqrt{N} (\log N)^{3/2 + \ve} \right) \qquad \textrm{as $N \to \infty$}
\end{equation}
for almost all $\alpha$, which holds for every strictly increasing sequence of positive integers $(b_k)_{k \geq 1}$ (see~\cite{baker}). It is known that in the general setting the upper bound given by~\eqref{baker} is optimal (up to powers of logarithms; see~\cite{bpt}). Thus the upper bound given in Theorem~\ref{th_thue2} is significantly stronger than the general metric discrepancy bound given by~\eqref{baker}. Furthermore we want to emphasize the fact that the precision of Theorem~\ref{th_thue2} is quite remarkable, in view of the fact that good bounds for the typical order of the discrepancy are only known for a very small number of classes of parametric sequences.\\

One of the main objectives of Theorems~\ref{th_thue} and~\ref{th_thue2} is to examine the degree of pseudorandomness of the parametric sequences $(\{n_k \alpha\})_{k \geq 1}$, and consequently also of the Thue--Morse sequence $(n_k)_{k \geq 1}$ of integers itself. By classical probability theory, for a sequence $X_1, X_2, \dots$ of independent, identically distributed (i.i.d.) random variables having uniform distribution on $[0,1]$ we have the law of the iterated logarithm (LIL)
\begin{equation} \label{salz}
\limsup_{N \to \infty} \frac{\left| \sum_{k=1}^N e^{2 \pi i h X_k} \right|}{\sqrt{2 N \log \log N}} = \frac{1}{\sqrt{2}} \qquad \textup{almost surely (a.s.)}
\end{equation}
and the Chung--Smirnov LIL for the Kolmogorov--Smirnov statistic (that is, for the discrepancy)
\begin{equation} \label{salz2}
\limsup_{N \to \infty} \frac{N D_N^*(X_1, \dots, X_N)}{\sqrt{2 N \log \log N}} = \frac{1}{2} \qquad \textup{a.s.};
\end{equation}
in other words, for a random sequence of points exponential sums are typically of asymptotic order roughly $\sqrt{N}$, and the discrepancy is typically also of the corresponding asymptotic order. Furthermore, similar results usually hold for exponential sums of $(r_k \alpha)_{k \geq 1}$ and for the discrepancy of $(\{r_k \alpha\})_{k \geq 1}$ when $(r_k)_{k \geq 1}$ is a ``random'' increasing sequence of integers. In the simplest model, when for every number $n \geq 1$ we decide independently and with fair probability whether it should be contained in $(r_k)_{k \geq 1}$ or not, then~\eqref{salz} holds almost surely (with respect to the probability space over which the $r_k$'s are defined) for almost all $\alpha$. In a similar fashion both results~\eqref{salz} and~\eqref{salz2} essentially remain valid when the random sequence $(r_k)_{k \geq 1}$ is constructed in a more complicated fashion (see for example~\cite{fu1,fu2,ras}).\\

Thus Theorem~\ref{th_thue} and Theorem~\ref{th_thue2} show that the typical\footnote{By a result which holds for ``typical'' $\alpha$ we mean a result which is valid for a set of full Lebesgue measure.} asymptotic order of exponential sums and of the discrepancy of $(\{n_k \alpha\})_{k \geq 1}$ for the Thue--Morse integers $(n_k)_{k \geq 1}$ does \emph{not} match with the corresponding order in the random case, by this means showing an interesting deviation from ``pseudorandom'' behavior of the sequence $(n_k)_{k \geq 1}$ itself. On the other hand, the behavior of $(\{n_k \alpha\})_{k \geq 1}$ also does \emph{not} match with the behavior of $n \alpha$-sequences for typical values of $\alpha$. More precisely, as already mentioned above, as a consequence of metric results of Khintchine~\cite{khin} and due to the fact that the discrepancy of $(\{n \alpha\})_{n \geq 1}$ can be expressed in terms of the continued fractions expansion of $\alpha$, we have
\begin{equation} \label{khin}
D_N^*( \{\alpha\}, \{2 \alpha\}, \dots, \{N \alpha\}) = \mathcal{O} \left( \frac{(\log N) (\log \log N)^{1+\ve})}{N} \right) \qquad \textrm{ as $N \to \infty$}
\end{equation}
for almost all $\alpha$. Consequently, by Theorem~\ref{th_thue2}, the typical asymptotic order of the discrepancy of parametric sequences $(\{n_k \alpha\})_{k \geq 1}$ is significantly larger than that of typical $n \alpha$-sequences, and by Theorem~\ref{th_thue} this is also true for exponential sums. Thus, with respect to exponential sums as well as with respect to the discrepancy, parametric sequences $(n_k \alpha)_{k \geq 1}$ generated by the Thue--Morse integers $(n_k)_{k \geq 1}$ occupy a position somewhere between $n \alpha$-sequences and truly random sequences. We also want to comment on the fact that there is a huge difference between the order of the exponential sums in Theorem~\ref{th_thue} and the order of the discrepancy in Theorem~\ref{th_thue2}. This is a very surprising phenomenon, which is related to problems from metric Diophantine approximation (which are implicit in the proof of Theorem~\ref{th_thue2}, and are briefly discussed in the concluding Section~\ref{sec_conc}).\\

As already mentioned earlier, it is rather difficult to give the right order for the exponential sums in Theorem~\ref{th_thue}, the trigonometric products in Theorem~\ref{th_theorem2}, and the discrepancy of $\left(\left\{n_{k} \alpha\right\}\right)_{k \geq 1}$ for concrete non-trivial examples of $\alpha$. What do we mean by a ``non-trivial'' example? In the first part of Section~\ref{sect_proof_th4} we will point out the following facts: 
\begin{itemize}
 \item The order of the discrepancy of the pure Kronecker sequence $\left(\left\{n \alpha\right\}\right)_{k \geq 1}$ never is significantly larger than the order of the  discrepancy of the Thue--Morse--Kronecker sequence $\left(\left\{n_{k} \alpha\right\}\right)_{k \geq 1}$.
 \item If the order of the discrepancy $D_{N}^{*}$ of the pure Kronecker sequence satisfies $N D_{N}^{*}= \Omega \left(N^{\frac{\log 3}{\log 4}}\right)$ then the discrepancy $\widetilde{D}_{N}^{*}$ of the Thue--Morse--Kronecker sequence is essentially of the same order as the discrepancy of the pure Kronecker sequence.
 \item If the order of the discrepancy $D_{N}^{*}$ of the pure Kronecker sequence satisfies $N D_{N}^{*} = \mathcal{O} \left(N^{\frac{\log 3}{\log 4}}\right)$ then $\widetilde{D}_{N}^{*}$ satisfies $N \widetilde{D}_{N}^{*} = \mathcal{O} \left(N^{\frac{\log 3}{\log 4} + \ve}\right)$.
\end{itemize}
~\\

Thus an ``interesting non-trivial'' example means for us an example where $\alpha$ is a ``natural'' real number such as $\sqrt{2}, e, \pi$ (it seems to us that there is no chance to handle these numbers since we do not have enough information on their digit representation), or where $\alpha$ is such that $\widetilde{D}_{N}^{*}$ and hence $D_{N}^{*}$ is small (say $N \widetilde{D}_{N}^{*} = \mathcal{O} \left(N^{\ve}\right)$ -- however we cannot give such examples) or where the quality of the distribution of the sequences $\left(\left\{n \alpha\right\}\right)_{n \geq 1}$ and $\left(\left\{n_{k} \alpha\right\}\right)_{k \geq 1}$ differ strongly. Two such examples are given in Theorem~\ref{th_thue4}. Especially in the first example the difference between $D_{N}^{*}$ and $\widetilde{D}_{N}^{*}$ is of the maximal possible form.\\

\begin{theorem} \label{th_thue4}
~\
\begin{enumerate}
\item [a)] Let $\alpha = \frac{2}{3} + \sum^{\infty}_{k=1} \frac{1}{4^{2^{k}}}$. Then for the star-discrepancy $D^{*}_{N}$ of the pure Kronecker sequence $\left(\left\{n \alpha\right\}\right)_{n \geq 1}$ we have
$$N D^{*}_{N} = \mathcal{O} \left(\log N\right),$$
whereas for the star-discrepancy $\widetilde{D}_{N}^{*}$ of the Thue--Morse--Kronecker sequence $\left(\left\{n_{k} \alpha\right\}\right)_{k \geq 1}$ we have
$$N \widetilde{D}_{N}^{*}= \mathcal{O} \left(N^{\frac{\log 3}{\log 4}+ \ve}\right) \mbox{and}$$
$$N \widetilde{D}_{N}^{*}= \Omega \left(N^{\frac{\log 3}{\log 4} - \ve}\right)$$
for every $\ve > 0.$
\item [b)] Let $\gamma=0.1001011001101001\ldots$ the Thue--Morse real in base 2. Then for the star-discrepancy $D^{*}_{N}$ of the pure Kronecker sequence $\left(\left\{n \gamma\right\}\right)_{n \geq 1}$ we have
$$N D_{N}^{*} = \mathcal{O} \left(N^{\ve}\right) \quad \mbox{for all}~ \ve > 0,$$
whereas for the star-discrepancy $\widetilde{D}_{N}^{*}$ of the Thue--Morse--Kronecker sequence $\left(\left\{n_{k} \gamma\right\}\right)_{k \geq 1}$ we have
$$N \widetilde{D}_{N}^{*} = \Omega \left(N^{0.6178775}\right).$$
\end{enumerate}
\end{theorem}
~\\

We would like to point out here that there is an intimate connection between distribution properties of $\left(\left\{n_{k} \alpha\right\}\right)_{k \geq 1}$ and of certain types of hybrid sequences. For some information on the analysis of hybrid sequences see for example~\cite{hokr},~\cite{hola} and~\cite{larch}.\\

As already mentioned, the proofs of Theorems~\ref{th_thue} and~\ref{th_thue2} are based on a connection between exponential sums of $(n_k \alpha)_{k \geq 1}$ and the lacunary trigonometric products studied in Theorem~\ref{th_theorem2}. We will establish this connection in the following lines, and exploit it in Section~\ref{sect_lac} in more detail. For the time being, we assume that $N$ is of the form $2^L$ for some positive integer $L$.\\

To analyze the exponential sums appearing in Theorem~\ref{th_thue} and on the right-hand side of~\eqref{erdtur}, we define
$$
S_h(N) = \sum_{k=1}^N e^{2 \pi i h n_k \alpha}.
$$
By the assumption that $N=2^L$ we have
\begin{eqnarray}
S_h(N) & = & \sum_{n=0}^{2N-1} \frac{1}{2} \sum_{\delta \in \{0,1\}} \exp \left(2 \pi i \delta \frac{s_2(n)}{2} \right) \exp \left(2 \pi i h n \alpha\right) \nonumber\\
& = & \frac{1}{2} \sum_{\delta \in \{0,1\}} ~\sum_{(\eta_0, \dots, \eta_{L}) \in \{0,1\}^{L+1}} \exp \left( 2 \pi i \sum_{\ell=0}^{L} \left(\delta \frac{\eta_\ell}{2} + \eta_\ell h 2^\ell \alpha \right) \right) \nonumber\\
& = & \frac{1}{2} \sum_{\delta \in \{0,1\}} ~\prod_{\ell=0}^{L} ~\left(\sum_{\eta_\ell \in \{0,1\}} \exp \left( 2 \pi i \eta_\ell \left( \frac{\delta}{2} + h 2^\ell \alpha \right)\right)\right), \nonumber
\end{eqnarray}
which yields
\begin{equation}
|S_h(N) |  \leq  \frac{1}{2} \prod_{\ell=0}^{L} |2 \sin \pi h 2^\ell \alpha| + \frac{1}{2} \prod_{\ell=0}^{L}| 2 \cos \pi h 2^\ell \alpha| \label{lacproducts}
\end{equation}
and
\begin{equation}
|S_h(N) | \geq   \left|\frac{1}{2} \prod_{\ell=0}^{L} |2 \sin \pi h 2^\ell \alpha| - \frac{1}{2} \prod_{\ell=0}^{L}| 2 \cos \pi h 2^\ell \alpha|\right|.\label{lacproductsLB}
\end{equation}

A similar analysis for the sequence $(m_k)_{k \geq 1}$ shows that
\begin{equation}
\left|\sum_{k=1}^N e^{2 \pi i h m_k \alpha}\right| \leq  \frac{1}{2} \prod_{\ell=0}^{L} |2 \sin \pi h 2^\ell \alpha| + \frac{1}{2} \prod_{\ell=0}^{L} |2 \cos \pi h 2^\ell \alpha|, \label{lacproducts2}
\end{equation}
where again we assume that $N=2^L$.\\

By taking logarithms, we can convert the trigonometric products appearing in~\eqref{lacproducts} and~\eqref{lacproducts2} into so-called \emph{lacunary sums}; these sums have been intensively investigated in Fourier analysis, and a wide range of mathematical methods is available for studying them (see the following Section~\ref{sect_lac}). Thus the theory of lacunary sums allows us to obtain an estimate for the size of the exponential sums $S_{h} (N)$ in the case when $N$ is a power of 2; however, it will turn out that we may also drop the condition that $N$ is an integral power of $2$ by applying a dyadic decomposition method.\\

Note that by~\eqref{khin} and by the fact that the Thue--Morse integers have asymptotic density 1/2 it is easy to show that all the conclusions of Theorem~\ref{th_thue} and Theorem~\ref{th_thue2} remain valid if we replace the sequence $(n_k)_{k \geq 1}$ by the sequence $(m_k)_{k \geq 1}$ (of those numbers which are \emph{not} Thue--Morse integers).\\

The outline of the remaining part of this paper is as follows. In Section~\ref{sect_lac} we explain the main principles of the theory of lacunary (trigonometric) sums, and state several lemmas as well as Theorem~\ref{lac_co1}, which we require for the proofs of Theorem~\ref{th_thue} and Theorem~\ref{th_theorem2}. In Section~\ref{sect_lac2} we give the proofs for the results stated in Section~\ref{sect_lac}, and in Section~\ref{sect_proof} we give the proofs of Theorem~\ref{th_thue} and Theorem~\ref{th_theorem2}. In Section~\ref{sect_proof_th2}, we prove Theorem~\ref{th_thue2}, and in Section~\ref{sect_proof_th4} we prove Theorem~\ref{th_thue4}. Finally, in Section~\ref{sec_conc}, we briefly mention a problem from metric Diophantine approximation, which was posed by LeVeque in~\cite{LeVe} and is related to the proof of Theorem~\ref{th_thue2}.

\section{Probabilistic results for lacunary trigonometric products} \label{sect_lac}

It is a well-known fact that so-called \emph{lacunary} systems of trigonometric
functions, that is, systems of the form $(\cos 2 \pi s_\ell \alpha)_{\ell \geq 1}$ or
$(\sin 2 \pi s_\ell \alpha)_{\ell \geq 1}$ for rapidly increasing $(s_\ell)_{\ell \geq 1}$,
exhibit properties which are typical for sequences of independent random variables. This similarity includes the central limit theorem, the law of the iterated logarithm, and Kolmogorov's ``Three series'' convergence theorem. The situation is particularly well understood when $(s_\ell)_{\ell \geq 1}$ satisfies the \emph{Hadamard gap condition} 
\begin{equation} \label{had}
\frac{s_{\ell+1}}{s_\ell} \geq q > 1, \qquad \ell \geq 1.
\end{equation}
To a certain degree this almost-independence property extends to systems $(f(s_\ell \alpha))_{\ell \geq 1}$ for a function $f$ which is periodic with period one and satisfies certain regularity properties; however, in this case the number-theoretic properties of $(s_\ell)_{\ell \geq 1}$ play an important role, and the almost-independent behavior generally fails when~\eqref{had} is relaxed to a weaker growth condition. The case which has been investigated in the greatest detail is that when $f$ has bounded variation on $[0,1]$, since this case is (by Koksma's inequality) closely connected to the discrepancy of the sequence of fractional parts $(\{s_\ell \alpha\})_{\ell \geq 1}$, which in turn can be interpreted as the (one-sided) Kolmogorov--Smirnov statistic adopted to the case of the uniform measure on $[0,1]$.\\

To estimate the trigonometric products appearing in~\eqref{lacproducts} and~\eqref{lacproducts2} we will use the equalities
\begin{equation} \label{prod_lac_1}
\prod_{\ell=0}^{L-1} \left| 2 \sin \pi h 2^\ell \alpha \right| = \exp \left( \sum_{\ell=0}^{L-1} \log \left| 2 \sin \pi h 2^\ell \alpha \right| \right)
\end{equation}
and
\begin{equation} \label{prod_lac_2}
\prod_{\ell=0}^{L-1} \left| 2 \cos \pi h 2^\ell \alpha \right| = \exp \left( \sum_{\ell=0}^{L-1} \log \left| 2 \cos \pi h 2^\ell \alpha \right| \right),
\end{equation}
respectively, to transform the problem of lacunary trigonometric products into a problem concerning lacunary sums. However, the functions 
\begin{equation} \label{f_1}
f_1 (\alpha) := \log \left| 2 \sin \pi \alpha \right|
\end{equation}
and 
\begin{equation} \label{f_2}
f_2 (\alpha) := \log \left| 2 \cos \pi \alpha \right|
\end{equation}
do \emph{not} have bounded variation in $[0,1]$ (see the figures below). Consequently, the known results are not applicable in this situation, and we have to adopt the proof techniques in such a way that they can handle this kind of problem.\footnote{There exist a few results concerning lacunary series when $f$ is neither required to have bounded variation, nor to be Lipschitz- or H\"older-continuous, nor to have a modulus of continuity of a certain regularity; see for example~\cite{matsu}. However, for these results the growth requirements for $(s_\ell)_{\ell \geq 1}$ are much stronger than~\eqref{had}, which means that they are not applicable in our case, since by~\eqref{prod_lac_1} and~\eqref{prod_lac_2} we have to deal with lacunary sequences growing exactly with the speed presumed in~\eqref{had}, and 
not faster.}\\

\begin{minipage}{\textwidth}
\begin{center}
\includegraphics[angle=0,width=70mm]{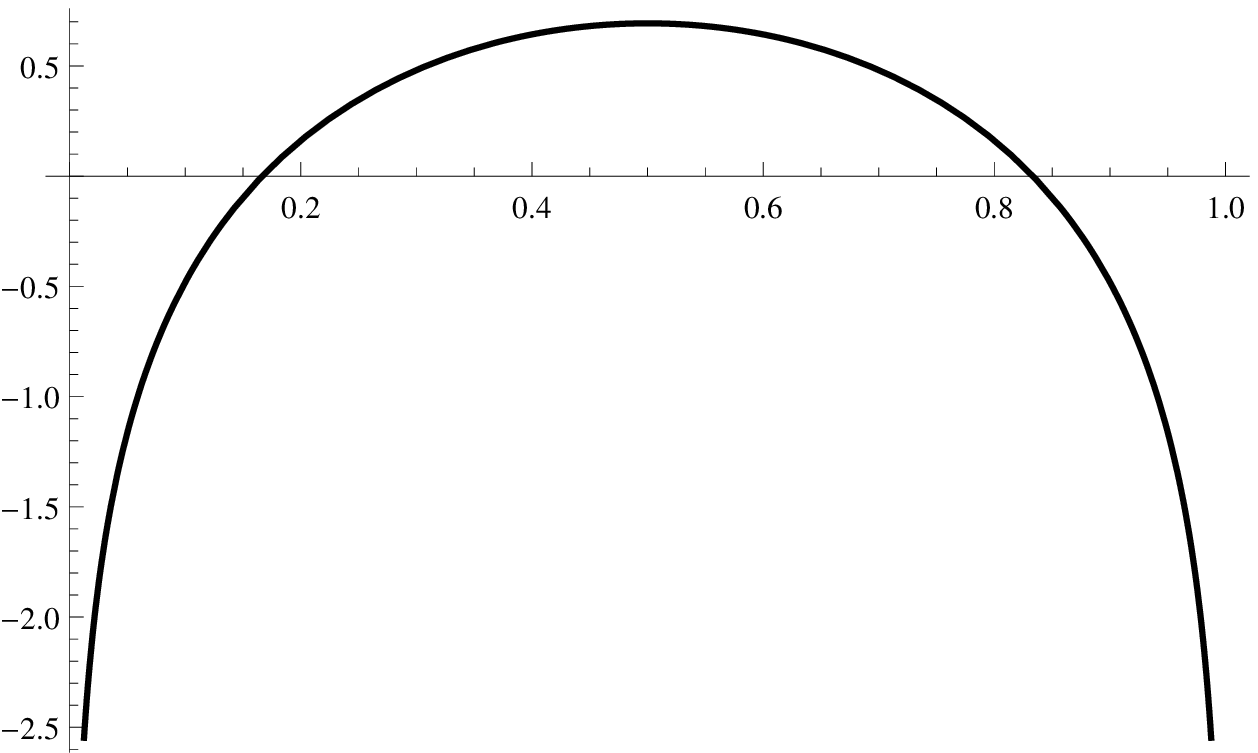} \hspace{1cm} \includegraphics[angle=0,width=70mm]{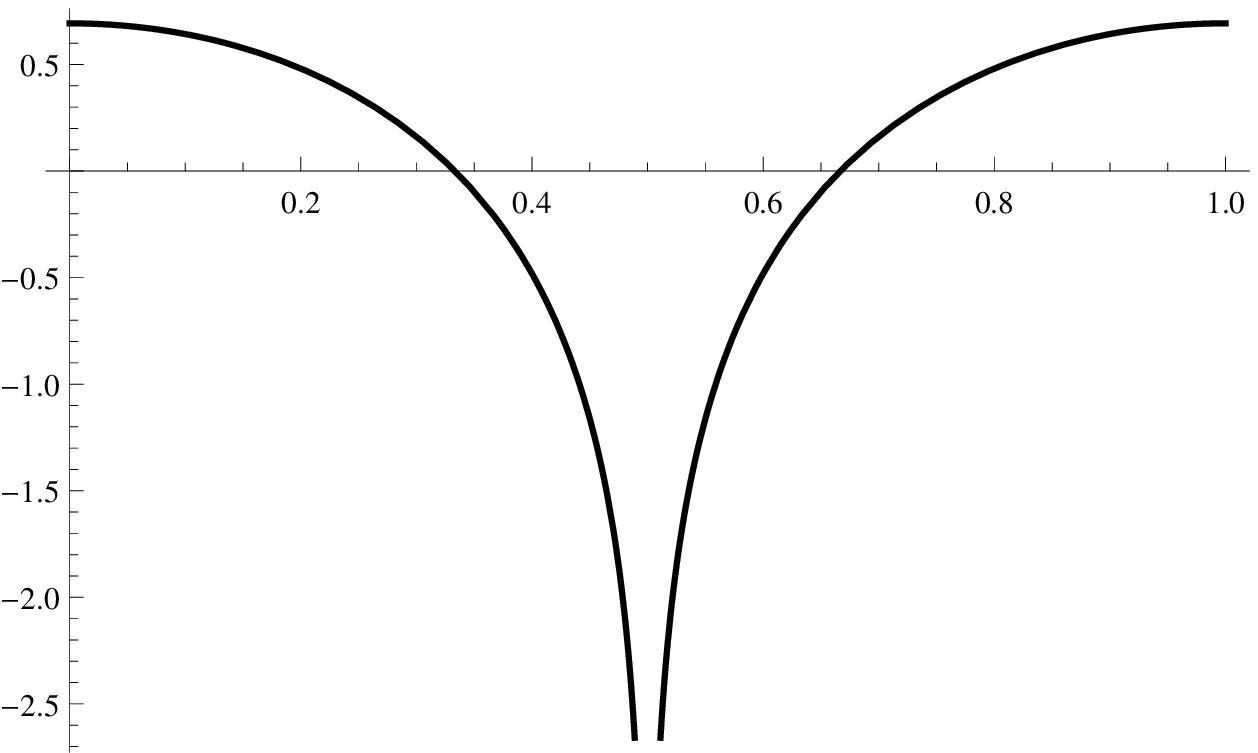}\\
Figure 1: the functions $f_1$ (left) and $f_2$ (right).
\end{center}
\end{minipage} \vspace{.3cm}

In the following we will assume that $f$ is a measurable real function satisfying
\begin{equation} \label{f}
f(\alpha+1)=f(\alpha), \qquad \int_0^1 f(\alpha) ~d\alpha = 0, \qquad \int_0^1 f(\alpha)^2  ~d\alpha < \infty,
\end{equation}
and that $(s_\ell)_{\ell \geq 1}$ is a sequence satisfying~\eqref{had}. We write 
$$
f(\alpha) \sim \sum_{j=1}^\infty a_j \cos 2 \pi j \alpha + b_j \sin 2 \pi j \alpha
$$
for the Fourier series of $f$, and we will assume that the Fourier coefficients of $f$ satisfy
\begin{equation} \label{f_coeff_b}
|a_j| \leq \frac{1}{j}, \qquad |b_j| \leq \frac{1}{j}, \qquad \textrm{for $j \geq 1$}.
\end{equation}
The inequalities in line~\eqref{f_coeff_b} appear frequently in the theory of lacunary series, since the upper bound stated there describes precisely (up to multiplication with a constant) the maximal asymptotic order of the Fourier coefficients of a function of bounded variation (see for example~\cite[p. 48]{zyg}). However, even if the function is \emph{not} of bounded variation the estimates in~\eqref{f_coeff_b} may still be true; this can be seen by the fact that for the (unbounded) functions $f_1$ and $f_2$ from lines~\eqref{f_1} and~\eqref{f_2}, respectively, we have
$$
f_1 (\alpha) \sim \sum_{j=1}^{\infty} \frac{-1}{j} \cos 2 \pi j \alpha 
$$
and
$$
f_2 (\alpha) \sim \sum_{j=1}^{\infty} \frac{(-1)^{j+1}}{j} \cos 2 \pi j \alpha.
$$
By the way, we note that both $f_1$ and $f_2$ are even functions, and that both of them satisfy~\eqref{f}.\\

Throughout the remaining part of this paper, we will write $\p$ for the Lebesgue measure on the unit interval. Note that the unit interval, equipped with Borel sets and Lebesgue measure, is a probability space, and that accordingly every measurable function on $[0,1]$ can be seen as a random variable over this probability space. We also write $\|\cdot \|_2$ for the $L^2 (0,1)$ norm and $\| \cdot \|_\infty$ for the supremum norm of a function, respectively.\\

The main technical tool in this section is the following exponential inequality (Lemma~\ref{lac_lemma1}). Together with the subsequent lemmas it will allow us to give an upper bound for the measure of those $\alpha$ for which~\eqref{prod_lac_1} and~\eqref{prod_lac_2} are large (stated in Lemma~\ref{lac_lemma3}). We state Lemmas~\ref{lac_lemma1}--\ref{lac_lemma3} in a slightly more general form than necessary for the proofs of Theorem~\ref{th_thue} and~\ref{th_theorem2}, and we will use them to prove an additional new theorem, namely Theorem~\ref{lac_co1} below.

\begin{lemma} \label{lac_lemma1}
Assume that $f$ is an even measurable function satisfying~\eqref{f}, whose Fourier coefficients satisfy~\eqref{f_coeff_b}. Furthermore, let $(s_\ell)_{\ell \geq 1}$ be a sequence of positive integers satisfying~\eqref{had} for some number $q > 1$. Then there exists a number $L_0 = L_0(q)$, such that the following holds. Let $L \geq L_0$ be given, and write $p(\alpha)$ for the $L^{8}$-th partial sum of the Fourier series of $f$. Then for all 
\begin{equation} \label{lambda}
\lambda \in \left[0, L^{-1/9} \right]
\end{equation}
we have
$$
\int_0^1 \exp\left(\lambda \sum_{\ell=1}^{L} p(s_\ell \alpha) \right) d\alpha \leq \exp \left( \frac{2 \lambda^2 \pi^2 L}{3} \right)
$$
and
$$
\int_0^1 \exp\left(\lambda \sum_{\ell=1}^{L} p(s_\ell \alpha) \right) d\alpha \leq \exp \left(12 \lambda^2 \frac{\sqrt{q}}{\sqrt{q}-1} \|p\|_2^{1/2} L \right).
$$
The same two conclusions hold if $f$ is an odd function instead of an even function.
\end{lemma}
We emphasize the fact that the number $L_0$ in the statement of Lemma~\ref{lac_lemma1} depends \emph{only} on the growth parameter $q$; it does not depend on the function $f$ or the sequence $(s_\ell)_{\ell \geq 1}$. The same will be true for the numbers $L_0(q)$ in Lemmas~\ref{lac_lemma1a} and~\ref{lac_lemma3} below.\\

From Lemma~\ref{lac_lemma1} we will deduce the following Lemma~\ref{lac_lemma1a}, which is a large deviations bound for the maximal partial sum of a lacunary sums.

\begin{lemma} \label{lac_lemma1a}
Let $f$ and $(s_\ell)_{\ell \geq 1}$ be as in Lemma~\ref{lac_lemma1}. Then there exists a number $L_0 = L_0(q)$, such that the following holds. Let $L \geq L_0$ be given, and assume that $L$ is an integral power of 2. Write $p(\alpha)$ for the $L^{8}$-th partial sum of the Fourier series of $f$. Then we have
\begin{eqnarray*}
\p \left( \alpha \in (0,1):~\max_{1 \leq M \leq L} \left| \sum_{\ell=1}^M p(s_\ell \alpha) \right| > \frac{29 q}{q-1} \sqrt{L \log \log L} \right) & \leq & \frac{48}{(\log L)^{1.4}}
\end{eqnarray*}
and, under the additional assumption that $\|p\|_2^{1/4} \geq L^{-1/100}$, we also have
\begin{eqnarray*}
\p \left( \alpha \in (0,1):~\max_{1 \leq M \leq L} \left| \sum_{\ell=1}^M p(s_\ell \alpha) \right| > 43 \|p\|_2^{1/4} \frac{\sqrt{q}}{\sqrt{q}-1} \sqrt{L} \sqrt{\log \log L} + \sqrt{L} \right) \leq \frac{45}{(\log L)^2}.
\end{eqnarray*}
The same two conclusions hold if $f$ is an odd function instead of an even function.
\end{lemma}

\begin{lemma} \label{lac_lemma2}
Let $f$ and $(s_\ell)_{\ell \geq 1}$ be as in Lemma~\ref{lac_lemma1}. Let $L$ be given, and write $r(\alpha)$ for the remainder term of the $L^{8}$-th partial sum of the Fourier series of $f$. Then we have
$$
\int_0^1 \left( \max_{1 \leq M \leq L} \left| \sum_{\ell=1}^{M} r(s_\ell \alpha) \right| \right)^2 ~d\alpha \leq \frac{4}{L^4}
$$
\end{lemma}

From Lemmas~\ref{lac_lemma1a} and~\ref{lac_lemma2} we will deduce the following Lemma~\ref{lac_lemma3}.

\begin{lemma} \label{lac_lemma3}
Let $f$ and $(s_\ell)_{\ell \geq 1}$ be as in Lemma~\ref{lac_lemma1}. Then there exists a number $L_0 = L_0(q)$, such that the following holds. Let $L \geq L_0$ be given, and assume that $L$ is an integral power of 2. Then we have
\begin{eqnarray*}
\p \left( \alpha \in (0,1):~\max_{1 \leq M \leq L} \left| \sum_{\ell=1}^M f(s_\ell \alpha) \right| > \left( \frac{30 q}{q-1} \right) \sqrt{L \log \log L} \right) \leq \frac{49}{(\log L)^{1.4}}
\end{eqnarray*}
and, under the additional assumption that $\|p\|_2^{1/4} \geq L^{-1/100}$, we also have
\begin{eqnarray*}
\p \left( \alpha \in (0,1):~\max_{1 \leq M \leq L} \left| \sum_{\ell=1}^M p(s_\ell \alpha) \right| > 43 \|p\|_2^{1/4} \frac{\sqrt{q}}{\sqrt{q}-1} \sqrt{L} \sqrt{\log \log L} + 2 \sqrt{L} \right) \leq \frac{46}{(\log L)^2}.
\end{eqnarray*}
The same two conclusions hold if $f$ is an odd function instead of an even function.
\end{lemma}

As a consequence of Lemma~\ref{lac_lemma3} we obtain the following theorem, which is a bounded law of the iterated logarithm and is of some interest in its own right. As far as we know, this is the first law of the iterated logarithm for Hadamard lacunary function series which can be applied to a class of unbounded functions $f$.\\

\begin{theorem} \label{lac_co1}
Assume that $f$ is an even measurable function satisfying~\eqref{f}, whose Fourier coefficients satisfy~\eqref{f_coeff_b}. Furthermore, let $(s_\ell)_{\ell \geq 1}$ be a sequence of positive integers satisfying~\eqref{had} for some number $q > 1$. Then we have
$$
\limsup_{L \to \infty} \frac{\left| \sum_{\ell=1}^L f(s_\ell \alpha) \right|}{\sqrt{L \log \log L}} \leq c_q
$$
for almost all $\alpha \in (0,1)$, where we can choose
$$
c_q = \max \left\{\frac{85 q}{q-1},  122 \|f\|_2^{1/4} \frac{\sqrt{q}}{\sqrt{q}-1} \right\}.
$$
\end{theorem}
~\\

We note in passing that from our proofs it seems that the conclusion of Theorem~\ref{lac_co1} remains true if the conditions $|a_j| \leq j^{-1}, ~|b_j| \leq j^{-1}$ in~\eqref{f_coeff_b} are relaxed to $|a_j| \leq j^{-1/2-\ve},~|b_j| \leq j^{-1/2-\ve}$ for some fixed $\ve > 0$; however, in this case the constant $c_q$ has to be replaced by some other constant $c_{q,\ve}$ which may also depend on $\ve$. We will not pursue this possible generalization any further in the present paper.\\

For the proofs of Theorem~\ref{th_thue} and~\ref{th_theorem2} we will also need the following result. It has first been stated by Fortet~\cite{fortet}; a concise proof can be found in~\cite{maruyama}. This result can be seen as a special case of the more general results in~\cite{aist2}. 

\begin{lemma} \label{lac_lemma4}
Let $f$ be a function satisfying~\eqref{f}, which additionally satisfies a H\"older continuity condition of order $\beta$ for some $\beta>0$. Then
$$
\limsup_{L \to \infty} \frac{\sum_{\ell=0}^{L-1} f(2^\ell \alpha)}{\sqrt{2 L \log \log L}} = \sigma_f \qquad \textrm{for almost all $\alpha$},
$$
where 
\begin{equation} \label{sigmap}
\sigma_f^2 = \lim_{m \to \infty} \frac{1}{m} \int_0^1 \left(f(\alpha) + \dots + f(2^{m-1} \alpha)\right)^2 d\alpha.
\end{equation}
\end{lemma}



\section{Proofs of results from Section~\ref{sect_lac}} \label{sect_lac2}

\begin{proof}[Proof of Lemma~\ref{lac_lemma1}] 
The proof of Lemma~\ref{lac_lemma1}, as well as the proofs of Lemmas~\ref{lac_lemma1a},~\ref{lac_lemma2},~\ref{lac_lemma3} and Theorem~\ref{lac_co1}, uses methods of Takahashi~\cite{taka} and Philipp~\cite{philipp}.\\

Assuming that $f$ is even, the $L^{8}$-th partial sum of the Fourier series of $f$ is of the form
$$
p(\alpha) = \sum_{j=1}^{L^8} a_j \cos 2 \pi j \alpha,
$$
where by assumption the coefficients $a_j$ satisfy the inequality on the left-hand side of~\eqref{f_coeff_b}. We note that~\eqref{f_coeff_b} implies that
\begin{equation} \label{pinf}
\|p\|_\infty \leq \sum_{j=1}^{L^8} \frac{1}{j} \leq 1+ 8\log L.
\end{equation}
We divide the set of integers $\{1, \dots, L\}$ into blocks $\Delta_1, \dots, \Delta_w$ of consecutive numbers, for some appropriate $w$, such that every block contains $\left\lceil \log_q \left(4 L^8\right) \right\rceil$ numbers (the last block may contain less).\footnote{We assume throughout that $L$ is sufficiently large such that all appearing logarithms are well-defined and positive.} More precisely, we set 
$$
w = \left\lceil \frac{L}{\lceil \log_q \left(4 L^8\right) \rceil} \right\rceil
$$ 
and
$$
\Delta_i = \Big\{ (i-1) \left\lceil \log_q \left(4 L^8\right) \right\rceil + 1, \dots, i \left\lceil \log_q (4 L^8) \right\rceil \Big\} \cap \{1, \dots, L\}, \qquad \textrm{for $1 \leq i \leq w$}.
$$
We set
$$
I_1 = \int_0^1 \exp \left( 2 \lambda \sum_{\substack{1 \leq i \leq w,\\ i ~\textrm{even}}} ~\sum_{\ell \in \Delta_i} p(s_\ell \alpha) \right) ~d\alpha
$$
and
$$
I_2 = \int_0^1 \exp \left( 2 \lambda \sum_{\substack{1 \leq i \leq w,\\ i ~\textrm{odd}}} ~\sum_{\ell \in \Delta_i} p(s_\ell \alpha) \right) ~d\alpha.
$$
Then by the Cauchy--Schwarz inequality we have
\begin{equation} \label{i1i2}
\int_0^1 \exp\left(\lambda \sum_{\ell=1}^{L} p(s_\ell \alpha) \right) ~d\alpha \leq \sqrt{I_1 I_2}.
\end{equation}
Writing 
\begin{equation*} \label{ui}
U_i = \sum_{\ell \in \Delta_i} p(s_\ell \alpha), \qquad 1 \leq i \leq w,
\end{equation*}
and using the inequality
$$
e^x \leq 1 + x + x^2, \qquad \textrm{which is valid for $|x| \leq 1$}, 
$$
we have
\begin{eqnarray}
I_1 & = & \int_0^1 \prod_{\substack{1 \leq i \leq w,\\ i ~\textrm{even}}} \exp \left( 2 \lambda \sum_{\ell \in \Delta_i} p(s_\ell \alpha) \right) ~d\alpha \nonumber\\
& \leq & \int_0^1 \prod_{\substack{1 \leq i \leq w,\\ i ~\textrm{even}}} \left( 1 + 2 \lambda U_i + 4 \lambda^2 U_i^2 \right) ~d\alpha, \label{i1alpha}
\end{eqnarray}
where we used the fact that by~\eqref{f_coeff_b},~\eqref{lambda} and~\eqref{pinf} we have
\begin{eqnarray*}
|2 \lambda U_i| & \leq & 2 L^{-1/9} \|p\|_\infty |\Delta_i| \\
& \leq & 2 L^{-1/9} (1 + 8 \log L) \left\lceil \log_q (4 L^8) \right\rceil\\
& \leq & 1
\end{eqnarray*}
for $L \geq L_0(q)$\\

Using the classical trigonometric identity
\begin{equation}\label{trigid}
(\cos y)(\cos z) = \frac{\cos (y+z) + \cos (y-z)}{2}, \qquad \textrm{ for $y,z \in \R$,}
\end{equation}
we have
\begin{eqnarray}
U_i^2 & = & \left( \sum_{\ell \in \Delta_i} ~\sum_{j=1}^{L^8} a_j \cos 2 \pi j s_\ell \alpha \right)^2 \nonumber\\
& = & \sum_{\ell_1,\ell_2 \in \Delta_i} ~\sum_{1 \leq j_1, j_2 \leq L^8} a_{j_1} a_{j_2} \frac{\cos ( 2 \pi (j_1 s_{\ell_1} + j_2 s_{\ell_2})) + \cos ( 2 \pi (j_1 s_{\ell_1} - j_2 s_{\ell_2}))}{2} \label{frequ}\\
& =: & V_i + W_i. \label{viwi}
\end{eqnarray}
Here we write $V_i$ for the sum of all those cosine-functions having frequencies in the interval $\left[ s_{\min(\Delta_i)}, 2 L^8 s_{\max (\Delta_i)}\right]$, where $\min(\Delta_i)$ and $\max(\Delta_i)$ denote the smallest resp. largest element of $\Delta_i$, and we write $W_i$ for the sum of those cosine-functions having frequencies smaller than $s_{\min(\Delta_i)}$. It is easy to check that no other frequencies can occur in~\eqref{frequ}. We note that all the frequencies of the cosine-functions in $U_i$ are also contained in the interval $\left[s_{\min(\Delta_i}, 2 L^8 s_{\max(\Delta_i)} \right]$, and write
\begin{equation} \label{xiuivi}
X_i = 2 \lambda U_i + 4 \lambda^2 V_i.
\end{equation}
Using this notation we have
$$
\prod_{\substack{1 \leq i \leq w,\\ i ~\textrm{even}}} \left( 1 + 2 \lambda U_i + 4 \lambda^2 U_i^2 \right) = \prod_{\substack{1 \leq i \leq w,\\ i ~\textrm{even}}} \left( 1 + X_i + 4 \lambda^2 W_i \right).
$$
From Minkowski's inequality and~\eqref{f_coeff_b} we deduce that
\begin{eqnarray}
W_i & \leq & \frac{1}{2} \underbrace{\sum_{\ell_1, \ell_2 \in \Delta_i}  ~\sum_{1 \leq j_1,j_2 \leq L^8}}_{|j_1 s_{\ell_1} - j_2 s_{\ell_2}| < \min(\Delta_i)} |a_{j_1} a_{j_2}| \nonumber\\
& \leq & \sum_{\substack{\ell_1, \ell_2 \in \Delta_i,\\ \ell_1 \leq \ell_2}}  ~\sum_{\substack{1 \leq j_1,j_2 \leq L^8,\\ j_1 > j_2 s_{\ell_2} / s_{\ell_1} - 1}} |a_{j_1} a_{j_2}| \label{wi1}\\
& \leq & \sum_{\substack{\ell_1, \ell_2 \in \Delta_i,\\ \ell_1 \leq \ell_2}}  ~\sum_{\substack{1 \leq j_1,j_2 \leq L^8,\\ j_1 > j_2 s_{\ell_2} / s_{\ell_1} - 1}} \frac{1}{j_1 j_2} \nonumber\\
& \leq & \sum_{\substack{\ell_1, \ell_2 \in \Delta_i,\\ \ell_1 \leq \ell_2}} ~\sum_{j=1}^{L^8} \frac{1}{j} \underbrace{\frac{1}{\lceil j s_{\ell_2} / s_{\ell_1} - 1 \rceil}}_{\leq 2 s_{\ell_1} / (j s_{\ell_2})} \nonumber\\
& \leq & 2 \sum_{\substack{\ell_1, \ell_2 \in \Delta_i,\\ \ell_1 \leq \ell_2}} ~\sum_{j=1}^{L^8} \frac{s_{\ell_1}}{j^2 s_{\ell_2}} \nonumber\\
& \leq & 2 |\Delta_i| \frac{q}{q-1} \frac{\pi^2}{6}.\label{wi2}
\end{eqnarray}
Another way of continuing from line~\eqref{wi1} is to use the Cauchy--Schwarz inequality, which leads to
\begin{eqnarray}
W_i & \leq & \sum_{\substack{\ell_1, \ell_2 \in \Delta_i,\\ \ell_1 \leq \ell_2}} \sum_{1 \leq j_2 \leq L^8} \frac{1}{j_2} \sum_{\substack{1 \leq j_1 \leq L^8, \\j_1 > j_2 s_{\ell_2} / s_{\ell_1} - 1}} |a_{j_1}| \nonumber\\
& \leq & \sum_{\substack{\ell_1, \ell_2 \in \Delta_i,\\ \ell_1 \leq \ell_2}}   \sum_{1 \leq j_2 \leq L^8} \frac{1}{j_2} \left(\underbrace{\sum_{\substack{1 \leq j_1 \leq L^8, \\j_1 > j_2 s_{\ell_2} / s_{\ell_1} - 1}}  \overbrace{a_{j_1}^2}^{\leq 1/j_1^2}}_{\leq 4 s_{\ell_1} /(j_2 s_{\ell_2})} \right)^{1/2} \underbrace{\left(\sum_{\substack{1 \leq j_1 \leq L^8, \\j_1 > j_2 s_{\ell_2} / s_{\ell_1} - 1}}  a_{j_1}^2\right)^{1/2}}_{\leq \|p\|_2^{1/2}} \nonumber\\
& \leq & 2 \sum_{\substack{\ell_1, \ell_2 \in \Delta_i,\\ \ell_1 \leq \ell_2}}   \sum_{1 \leq j \leq L^8} \sqrt{\frac{s_{\ell_1}}{j^3 s_{\ell_2}}} \|p\|_2^{1/2} \nonumber\\
& \leq & 6 \frac{\sqrt{q}}{\sqrt{q}-1} \|p\|_2^{1/2}. \label{wi3}
\end{eqnarray}

Now assume that $i_1 < i_2$ are two indices from the set $\{1, \dots, w\}$, and that both $i_1$ and $i_2$ are even. Then by construction the frequency of any trigonometric function in $X_{i_2}$ is at least twice as large as the frequency of any trigonometric function in $X_{i_1}$. To see why this is the case, we recall that the frequency of the largest trigonometric function in $X_{i_2}$ is at most $2 L^8 s_{\max(\Delta_{i_1})}$, that the frequency of the smallest trigonometric function in $X_{i_1}$ is at least $s_{\min(\Delta_{i_2})}$, and that by~\eqref{had}
\begin{eqnarray*}
\frac{s_{\min(\Delta_{i_2})}}{s_{\max(\Delta_{i_1})}} & \geq & q^{\min(\Delta_{i_2}) - \max(\Delta_{i_1})} \\
& = & q^{\left\lceil \log_q (4 L^8) \right\rceil}\\
& \geq & 4 L^8.
\end{eqnarray*}
As a consequence for every set of distinct indices $i_{1}, \dots, i_v$ (where the cardinality $v$ is arbitrary), all of which are even and are contained in $\{1, \dots, w\}$, the functions $X_{i_1}, \dots, X_{i_v}$ are orthogonal, i.e.
\begin{equation} \label{ortho}
\int_0^1 X_{i_1} \cdot \dots \cdot X_{i_v} ~d\alpha = 0
\end{equation}
(this argument is explained in more detail in~\cite{philipp,taka}). Thus by~\eqref{i1alpha},~\eqref{viwi},~\eqref{xiuivi},~\eqref{wi2} and~\eqref{ortho} we have
\begin{eqnarray*}
I_1 & \leq & \int_0^1 \prod_{\substack{1 \leq i \leq w,\\ i ~\textrm{even}}} \left( 1 + X_i + 4 \lambda^2 W_i \right)~d\alpha \\
& \leq & \int_0^1 \prod_{\substack{1 \leq i \leq w,\\ i ~\textrm{even}}} \left(1 + X_i + \frac{4 \lambda^2 |\Delta_i| \pi^2}{3} \frac{q}{q-1}\right) ~d\alpha \\
& = & \int_0^1 \underbrace{\prod_{\substack{1 \leq i \leq w,\\ i ~\textrm{even}}} \left(1 + \frac{4 \lambda^2 |\Delta_i| \pi^2}{3} \frac{q}{q-1}\right)}_{\textrm{does not depend on $\alpha$}} ~d\alpha \\
& \leq & \prod_{\substack{1 \leq i \leq w,\\ i ~\textrm{even}}} \exp \left(\frac{4 \lambda^2 |\Delta_i| \pi^2}{3} \frac{q}{q-1}\right)\\
& = & \exp \left( \sum_{\substack{1 \leq i \leq w,\\ i ~\textrm{even}}} \frac{4 \lambda^2 |\Delta_i| \pi^2}{3} \frac{q}{q-1}\right).
\end{eqnarray*}
In the same way we can get an upper bound for $I_2$, and thus by~\eqref{i1i2} we finally obtain
\begin{eqnarray*}
& & \int_0^1 \exp\left(\lambda \sum_{\ell=1}^{L} p(s_\ell \alpha) \right) ~d\alpha\\
& \leq & \exp \left( \sum_{\substack{1 \leq i \leq w,\\ i ~\textrm{even}}} \frac{2 \lambda^2 |\Delta_i| \pi^2}{3} \frac{q}{q-1}\right) \exp \left( \sum_{\substack{1 \leq i \leq w,\\ i ~\textrm{odd}}} \frac{2 \lambda^2 |\Delta_i| \pi^2}{3} \frac{q}{q-1}\right) \\
& = & \exp \left( \frac{2 \lambda^2 L \pi^2}{3} \frac{q}{q-1}\right).
\end{eqnarray*}
This proves the first conclusion of the lemma. In the same way we can use~\eqref{wi3} (and the corresponding upper bound for $I_2$) to obtain
$$
\int_0^1 \exp\left(\lambda \sum_{\ell=1}^{L} p(s_\ell \alpha) \right) ~d\alpha  \leq \exp \left(12 \lambda^2 \frac{\sqrt{q}}{\sqrt{q}-1} \|p\|_2^{1/2} \right),
$$
which proves the second conclusion of the lemma.\\

Thus we have proved both parts of Lemma~\ref{lac_lemma1} in the case when $f$ is even; the proof in the odd case can be carried out in exactly the same way.\\
\end{proof}

\begin{proof}[Proof of Lemma~\ref{lac_lemma1a}] 
By assumption $L$ is an integral power of 2. We set $\nu = \log_2 L$. By classical dyadic decomposition, we can write every subset $\{1, \dots, M\}$ of $\{1, \dots, L\}$ as the disjoint sum of at most one set of cardinality $2^{\nu-1}$, at most one set of cardinality $2^{\nu-2}$, at most one set of cardinality $2^{\nu-3}$, and so on, at most one set of cardinality $2^{\lceil \nu/4 \rceil}$, and additionally at most one set of cardinality \emph{at most} $2^{\lceil \nu/4 \rceil}$, where all these sets contain consecutive positive integers. To be able to represent every sets $\{1, \dots, M\}$ in this way, we need $2^{\mu}$ sets of cardinality $2^{\nu - \mu}$, for $\mu \in \{1, \dots, \nu-\lceil \nu/4 \rceil\}$, and all the sets of cardinality at most $2^{\lceil \nu/4 \rceil}$ starting at an integer multiple of $2^{\lceil \nu/4 \rceil}$. More precisely, the sets of cardinality $2^{\nu-\mu}$ are of the form
$$
\left\{j 2^{\nu-\mu}+1, \dots, (j+1) 2^{\nu-\mu} \right\}, \qquad j \in \{0, \dots, 2^{\mu}-1\}, \quad \mu \in \{1, \dots, \nu- \lceil \nu/4\rceil\},
$$
and the sets of cardinality at most $2^{\nu/4}$ are of the form 
$$
\left\{j 2^{\lceil \nu/4 \rceil} + 1, , \dots, j 2^{\lceil \nu/4 \rceil} + w \right\}, \qquad j \in \{0, \dots, 2^{\nu - \lceil \nu/4 \rceil} -1 \}, \quad w \in \left\{1, \dots, 2^{\lceil \nu/4 \rceil}\right\}.
$$
For $j \in \{0, \dots, 2^{\mu}-1\}$ and $\mu \in \left\{1, \dots, \nu - \lceil \nu/4 \rceil \right\}$, we set
$$
G_{\mu,j} = \left( \alpha \in (0,1):~\sum_{\ell=j 2^{\nu-\mu}+1}^{(j+1) 2^{\nu-\mu}} p(s_\ell \alpha) >  \frac{8 q}{q-1} \sqrt{2^{\nu-\mu}} \sqrt{\log \log 2^{\nu-\mu}} + \mu \sqrt{2^{\nu-\mu}} \right)
$$
Using the first part of Lemma~\ref{lac_lemma1} with 
$$
\lambda = \frac{\sqrt{\log \log (2^{\nu-\mu})}}{\sqrt{2^{\nu-\mu}}}
$$
(note that by our construction this value of $\lambda$ is admissible in Lemma~\ref{lac_lemma1}, provided that $L$ is sufficiently large) we have
$$
\int_0^1 \exp\left(\frac{\sqrt{\log \log (2^{\nu-\mu})}}{\sqrt{2^{\nu-\mu}}} \sum_{\ell=j 2^{\nu-\mu}+1}^{(j+1) 2^{\nu-\mu}} p(s_\ell \alpha) \right) ~d\alpha \leq \exp \left( \frac{2 \log \log (2^{\nu-\mu}) q \pi^2}{3(q-1)} \right),
$$
and consequently 
$$
\p (G_{\mu,j}) \leq \exp \left( \left(\frac{2 \pi^2}{3} - 8 \right) \frac{q}{q-1} \log \log (2^{\nu-\mu}) - \mu \right) \leq \frac{1}{e^{\mu} (\log (2^{\nu-\mu}))^{1.4}}.
$$
Thus we have
\begin{eqnarray}
\p \left(\bigcup_{\mu=1}^{\nu- \lceil \nu/4 \rceil} ~ \bigcup_{j=0}^{2^{\mu}-1} G_{\mu,j} \right) & \leq & \sum_{\mu=1}^{\nu - \lceil \nu/4 \rceil} \frac{2^\mu}{e^{\mu} \left(\log \left(2^{\lceil \nu/4 \rceil}\right)\right)^{1.4}} \nonumber\\
& \leq & \frac{20}{(\log L)^{1.4}}. \label{probg}
\end{eqnarray}
Now we set
$$
H_{j} = \left( \alpha \in (0,1):~\max_{1 \leq w \leq 2^{\lceil \nu/4} \rceil} \left| \sum_{\ell=j 2^{\lceil \nu/4 \rceil} + 1}^{j 2^{\lceil \nu/4 \rceil} + w} p(s_\ell \alpha) \right| >  \sqrt{L} \right), \qquad j \in \{0, \dots, 2^{\nu - \lceil \nu/4 \rceil} -1 \}.
$$
By~\eqref{f_coeff_b}, Minkowski's inequality, and the Carleson--Hunt inequality (see for example~\cite{arias}), we have
\begin{eqnarray}
& & \left( \int_0^1 \max_{1 \leq w \leq 2^{\lceil \nu/4 \rceil}} \left( \sum_{\ell=j 2^{\lceil \nu/4 \rceil} + 1}^{j 2^{\lceil \nu/4 \rceil}+w} p(s_\ell \alpha) \right)^6 d\alpha \right)^{1/6} \nonumber\\
& \leq & \sum_{j=1}^{L^8} |a_j| \left( \int_0^1 \max_{1 \leq w \leq 2^{\lceil \nu/4 \rceil}} \left( \sum_{\ell=j 2^{\lceil \nu/4 \rceil} + 1}^{j 2^{\lceil \nu/4 \rceil}+w} \cos 2 \pi j s_\ell \alpha \right)^6 d\alpha \right)^{1/6} \nonumber\\
& \leq & (1 + 8 \log L) c_{\textup{abs}} \left( \int_0^1  \left( \sum_{\ell=j 2^{\lceil \nu/4 \rceil} + 1}^{(j+1) 2^{\lceil \nu/4 \rceil}} \cos 2 \pi j s_\ell \alpha \right)^6 d\alpha \right)^{1/6} \label{lacint}
\end{eqnarray}
for some absolute constant $c_{\textup{abs}}$. Estimating the integral in~\eqref{lacint} can be reduced (via~\eqref{trigid}) to the problem of counting the number of solutions $(\ell_1, \dots, \ell_6)$ of the Diophantine equation
$$
s_{\ell_1} \pm \dots \pm s_{\ell_6} = 0, \qquad \textrm{for indices $\ell_1, \dots, \ell_6$ in the respective index range};
$$
we have
$$
\left( \int_0^1  \left( \sum_{\ell=j 2^{\lceil \nu/4 \rceil} + 1}^{(j+1) 2^{\lceil \nu/4 \rceil}} \cos 2 \pi j s_\ell \alpha \right)^6 d\alpha \right)^{1/6} \leq c_q \left( 2^{\nu/4} \right)^{1/2}
$$
for some constant $c_q$ depending only on $q$ (see, for example,~\cite{erdgal}). As a consequence by Markov's inequality we obtain 
$$
\p (H_j) \leq \hat{c}_q (1 + 8 \log L)^6 \frac{L^{6/8}}{L^3} = \hat{c}_q (1 + 8 \log L)^6 L^{-9/4}
$$
and
\begin{equation} \label{probh}
\p \left( \bigcup_{j=0}^{2^{\nu - \lceil \nu/4 \rceil} -1} H_j \right) \leq L^{3/4} \hat{c}_q (1 + 8 \log L)^6 L^{-9/4} = \hat{c}_q (1 + 8 \log L)^6 L^{-6/4}
\end{equation}
for some constant $\hat{c}_q$ depending only on $q$. We set
$$
F = \left( \bigcup_{\mu=1}^{\nu-\lceil\nu/4\rceil} \bigcup_{j=0}^{2^{\mu}-1} G_{\mu,j} \right) \cup \left( \bigcup_{j=0}^{2^{\nu - \lceil\nu/4\rceil} -1} H_j \right).
$$
Then by~\eqref{probg} and~\eqref{probh} we have
\begin{eqnarray} 
\p \left( F \right) & \leq & \frac{23}{(\log L)^{1.4}} + \hat{c}_q (1 + 8 \log L)^6 L^{-6/4} \\
& \leq &  \frac{24}{(\log L)^{1.4}} \label{Fprob}
\end{eqnarray}
for sufficiently large $L$.\\

By the dyadic decomposition described at the beginning of this proof, for every $\alpha \in F^C$ (where $F^C$ denotes the complement of the set $F$) we have
\begin{eqnarray*}
\sum_{\ell=1}^M p(s_\ell \alpha) & \leq & \sum_{\mu=1}^{\nu - \lceil\nu/4\rceil} \left( \frac{8q}{q-1} \sqrt{2^{\nu-\mu}} \sqrt{\log \log 2^{\nu-\mu}} + \mu \sqrt{2^{\nu-\mu}} \right) + \sqrt{L} \\
& \leq & \frac{8q}{q-1} (1 + \sqrt{2}) \sqrt{L \log \log L} + (4 + 3 \sqrt{2})  \sqrt{L} + \sqrt{L} \\
& \leq & \frac{29q}{q-1} \sqrt{L \log \log L}
\end{eqnarray*}
for all possible values $M \in \{1, \dots, L\}$. Thus we have shown that 
\begin{eqnarray} 
\p \left( \alpha \in (0,1):~\max_{1 \leq M \leq L} \left(\sum_{\ell=1}^M p(s_\ell \alpha)\right) > \frac{29 q}{q-1} \sqrt{L \log \log L} \right) \leq \frac{24}{(\log L)^{1.4}} \label{exprop}
\end{eqnarray}
for sufficiently large $L$. Note that whenever the function $f$ satisfies the assumptions of Lemma~\ref{lac_lemma1}, then the function $-f$ also satisfies these assumptions. Thus applying exactly the same arguments as above to the functions $-f$ and $-p$ instead of $f$ and $p$ we also obtain
\begin{eqnarray*}
\p \left( \alpha \in (0,1):~\max_{1 \leq M \leq L} \left(\sum_{\ell=1}^M \left(- p(s_\ell \alpha)\right)\right) > \frac{29 q}{q-1} \sqrt{L \log \log L} \right) \leq \frac{24}{(\log L)^{1.4}}
\end{eqnarray*}
for sufficiently large $L$, which, together with~\eqref{exprop}, proves the first conclusion of Lemma~\ref{lac_lemma1a}.\\

The proof of the second conclusions of Lemma~\ref{lac_lemma1a} is very similar to that of the first conclusion of the lemma. We use the same dyadic decomposition, but now we define 
$$
G_{\mu,j} = \left( \alpha \in (0,1):~\sum_{\ell=j 2^{\nu-\mu}+1}^{(j+1) 2^{\nu-\mu}} p(s_\ell \alpha) > \left(14 \|p\|_2^{1/4} \frac{\sqrt{q}}{\sqrt{q}-1} + \|p\|_2^{1/4} \mu \right) \sqrt{2^{\nu-\mu}} \sqrt{\log \log 2^{\nu-\mu}} \right)
$$
and use the second part of Lemma~\ref{lac_lemma1} with
$$
\lambda = \|p\|_2^{-1/4} \frac{\sqrt{\log \log (2^{\nu-\mu})}}{\sqrt{2^{\nu-\mu}}}
$$
Note that this choice of $\lambda$ is admissible, due to the restrictions that $\mu \leq \nu - \lceil \nu/4\rceil$ and $\|p\|_2^{1/4} \geq L^{-1/100}$. We obtain that
$$
\int_0^1 \exp\left(\|p\|_2^{-1/4} \frac{\sqrt{\log \log (2^{\nu-\mu})}}{\sqrt{2^{\nu-\mu}}} \sum_{\ell=j 2^{\nu-\mu}+1}^{(j+1) 2^{\nu-\mu}} p(s_\ell \alpha) \right) ~d\alpha \leq \exp \left( \frac{12 \sqrt{q}\log \log (2^{\nu-\mu})}{\sqrt{q}-1} \right),
$$
and consequently 
$$
\p (G_{\mu,j}) \leq \exp \left( \left(12 - 14 \right) \frac{\sqrt{q}}{\sqrt{q}-1} \log \log (2^{\nu-\mu}) - \mu \right) \leq \frac{1}{e^{\mu} (\log (2^{\nu-\mu}))^{2}}.
$$
The sets $H_j$ can be defined in the same way as in the proof of the first part of the lemma. Using similar calculations we obtain that
\begin{eqnarray*}
\max_{1 \leq L \leq M} \left| \sum_{\ell=1}^M p(s_\ell \alpha) \right| & \leq & \sum_{\mu=1}^\infty 2^{-\mu/2} \left(\left(14+\mu\right) \|p\|_2^{1/4} \frac{\sqrt{q}}{\sqrt{q}-1} \sqrt{L} \sqrt{\log \log L} \right) + \sqrt{L} \\
& \leq & 43 \|p\|_2^{1/4} \frac{\sqrt{q}}{\sqrt{q}-1} \sqrt{L} \sqrt{\log \log L} + \sqrt{L},
\end{eqnarray*}
except for a set of measure at most
$$
\left(\sum_{\mu=1}^\infty \frac{2^\mu}{e^\mu (\log (L^{1/4}))^2}\right) + \hat{c}_q (1 + 8 \log L)^6 L^{-9/4} \leq \frac{45}{(\log L)^2}
$$
(provided that $L$ is sufficiently large). This proves the second part of Lemma~\ref{lac_lemma1a}.\\
\end{proof}

\begin{proof}[Proof of Lemma~\ref{lac_lemma2}] 
The lemma follows from a simple application of Minkowski's inequality. We have
\begin{eqnarray}
\int_0^1 \left( \max_{1 \leq M \leq L} \left| \sum_{\ell=1}^{M} r(s_\ell \alpha) \right| \right)^2 d\alpha & \leq & \sum_{M=1}^L \int_0^1 \left(\left| \sum_{\ell=1}^{M} r(s_\ell \alpha) \right| \right)^2 d\alpha \nonumber\\
& \leq & \sum_{M=1}^L \left( \sum_{\ell=1}^M \|r\|_2  \right)^2. \label{rest}
\end{eqnarray}
By~\eqref{f_coeff_b} we have 
$$
\|r\|_2 \leq \left(\sum_{j=L^8+1}^{\infty} \frac{2}{j^2}\right)^{1/2} \leq \frac{2}{L^{7/2}},
$$
which implies, together with~\eqref{rest}, that
$$
\int_0^1 \left( \max_{1 \leq M \leq L} \left| \sum_{\ell=1}^{M} r(s_\ell \alpha) \right| \right)^2 ~d\alpha \leq \frac{4}{L^4}.
$$
This proves the lemma.\\
\end{proof}

\begin{proof}[Proof of Lemma~\ref{lac_lemma3}] 
Assume that we have decomposed $f=p+r$ as in Lemmas~\ref{lac_lemma1},~\ref{lac_lemma1a} and~\ref{lac_lemma2}. By Lemma~\ref{lac_lemma2} and Markov's inequality we have
$$
\p \left( \alpha \in (0,1):~\max_{1 \leq M \leq L} \left| \sum_{k=1}^M r(s_\ell \alpha) \right| > \sqrt{L} \right) \leq \frac{4}{L^5}.
$$
Together with Lemma~\ref{lac_lemma1a} this yields
\begin{eqnarray*}
& & \p \left( \alpha \in (0,1):~\max_{1 \leq M \leq L} \left| \sum_{k=1}^M f(s_\ell \alpha) \right| > \underbrace{\left(\frac{29 q}{q-1}+1\right)}_{\leq 30 q / (q-1)} \sqrt{L \log \log L} \right) \\
& \leq & \frac{48}{(\log N)^{1.4}} + \frac{4}{L^5} \\
& \leq & \frac{49}{(\log L)^{1.4}}
\end{eqnarray*}
for sufficiently large $L$, which is the first part of Lemma~\ref{lac_lemma3}.\\

In the same way we can deduce the second conclusion of Lemma~\ref{lac_lemma3} from a combination of the second conclusion of Lemma~\ref{lac_lemma1a} and Lemma~\ref{lac_lemma2}.\\
\end{proof}

\begin{proof}[Proof of Theorem~\ref{lac_co1}]
 From Lemma~\ref{lac_lemma1a} and Lemma~\ref{lac_lemma2} we can easily deduce Theorem~\ref{lac_co1}, using standard methods. Let us first assume that $f$ is either even or odd. For $m \geq 1$, let $E_m$ denote the sets defined by 
 $$
 E_{m} = \left\{ \alpha \in (0,1):~ \max_{1 \leq M \leq 2^m} \left| \sum_{\ell=1}^M f(s_\ell \alpha) \right| >  \frac{30 q}{q-1} \sqrt{2^m \log \log 2^m} \right\}.
 $$
 Then by Lemma~\ref{lac_lemma4} we have
 \begin{eqnarray}
 \p (E_m) & \leq & \frac{49}{\left(\log 2^m\right)^{1.4}} \label{prob2}
 \end{eqnarray}
 which implies that
 $$
 \sum_{m=1}^\infty \p (E_m) < \infty.
 $$
 Thus by the Borel--Cantelli lemma with probability 1 only finitely many events $E_m$ happen; in other words, for almost all $\alpha \in (0,1)$ we have
 $$
 \max_{1 \leq M \leq 2^m} \left| \sum_{\ell=1}^M f(s_\ell \alpha) \right|  \leq \frac{30 q}{q-1} \sqrt{2^m \log \log 2^m} \qquad \textrm{for $m \geq m_0(\alpha)$}.
 $$
As a consequence for almost all $\alpha \in (0,1)$ we have
\begin{equation} \label{lilf}
\limsup_{L \to \infty} \frac{\left| \sum_{\ell=1}^L f(s_\ell \alpha) \right|}{\sqrt{L \log \log L}} \leq \sqrt{2} \frac{30 q}{q-1},
\end{equation}
which proves Theorem~\ref{lac_co1} in the case when $f$ is either even or odd. For general $f$ we apply~\eqref{lilf} to the even and odd part separately, which results in an additional multiplicative factor of 2. Note that $2 \sqrt{2} \cdot 30 \leq 85$.\\

In the same way we can use the second conclusion of Lemma~\ref{lac_lemma3} to obtain
\begin{equation} \label{lilf2}
\limsup_{L \to \infty} \frac{\left| \sum_{\ell=1}^L f(s_\ell \alpha) \right|}{\sqrt{L \log \log L}} \leq \underbrace{2 \sqrt{2} \cdot 43}_{\leq 122} \|p\|_2^{1/4} \frac{\sqrt{q}}{\sqrt{q}-1},
\end{equation}
for almost all $\alpha \in (0,1)$. Theorem~\ref{lac_co1} now follows from a combination of~\eqref{lilf} and~\eqref{lilf2}.\\
\end{proof}

\section{Exponential Sums and Trigonometric Products: Proofs of Theorem~\ref{th_thue} and Theorem~\ref{th_theorem2}} \label{sect_proof}

Theorem~\ref{th_thue} and Theorem~\ref{th_theorem2} will follow easily from the following lemma, which is a version of Lemma~\ref{lac_lemma4} in the case when the function $f$ only satisfies~\eqref{f_coeff_b} (instead of being H\"older-continuous). We state it only for the special case of the two functions $f_1$ and $f_2$ from~\eqref{f_1} and~\eqref{f_2}.

\begin{lemma} \label{lemmath1}
We have
$$
\limsup_{L \to \infty} \frac{\sum_{\ell=0}^{L-1} f_1(2^\ell \alpha)}{\sqrt{2 L \log \log L}} = \frac{\pi}{\sqrt{2}}
$$
and
$$
\limsup_{L \to \infty} \frac{\left|  \sum_{\ell=0}^{L-1} f_2(2^\ell \alpha)\right|}{\sqrt{2 L \log \log L}} = 0
$$
for almost all $\alpha$.
\end{lemma}

\begin{proof}[Proof of Lemma~\ref{lemmath1}]
As already noted, the functions $f_1$ and $f_2$ satisfy conditions~\eqref{f} and~\eqref{f_coeff_b}. Let a number $d$ be given, and write $p_1$ for the $d$-th partial sum of the Fourier series of $f_1$, and $r_1$ for the remainder term. Then by Lemma~\ref{lac_lemma4} we have
\begin{equation} \label{LILp}
\limsup_{L \to \infty} \frac{\sum_{\ell=0}^{L-1} p_1(2^\ell \alpha)}{\sqrt{2 L \log \log L}} = \sigma_{p_1} \qquad \textrm{for almost all $\alpha$}
\end{equation}
where $\sigma_{p_1}$ is defined according to~\eqref{sigmap} (for the function $p_1$). Note that Lemma~\ref{lac_lemma4} is applicable since $p$ is a trigonometric polynomial (and consequently also is Lipschitz-continuous). Note furthermore that by~\eqref{f_coeff_b} we have
$$
\|r_1\|_2 \leq \left( \sum_{j=1}^{d+1} \frac{1}{j^2} \right)^{1/2} \leq d^{-1/2}.
$$
Consequently by Theorem~\ref{lac_co1} we have
\begin{equation} \label{LILr}
\limsup_{L \to \infty} \frac{\left| \sum_{\ell=0}^{L-1} r_1(2^\ell \alpha) \right|}{\sqrt{2 L \log \log L}} \leq 122 d^{-1/8} \frac{\sqrt{2}}{\sqrt{2}-1}  \qquad \textrm{for almost all $\alpha$}.
\end{equation}
Replacing $f_1$ by $f_2$ and replacing $p_1$ and $r_1$ by $p_2$ and $r_2$, respectively, we obtain~\eqref{LILp} and~\eqref{LILr} with $p_1$ and $r_1$ replaced by $p_2$ and $r_2$, respectively. Some standard calculations show that 
$$
\sigma_{p_1} \to \sigma_{f_1} \qquad \textrm{and} \qquad \sigma_{p_2} \to \sigma_{f_2} \qquad \textrm{as $d \to \infty$}.
$$
Furthermore, the expression on the right-hand side of~\eqref{LILr} clearly tends to zero as $d \to \infty$. Thus, overall we have
\begin{equation} \label{LILf1}
\limsup_{L \to \infty} \frac{\sum_{\ell=0}^{L-1} f_1(2^\ell \alpha)}{\sqrt{2 L \log \log L}} = \sigma_{f_1} \qquad \textrm{for almost all $\alpha$}
\end{equation}
and
\begin{equation} \label{LILf2}
\limsup_{L \to \infty} \frac{\sum_{\ell=0}^{L-1} f_2(2^\ell \alpha)}{\sqrt{2 L \log \log L}} = \sigma_{f_2} \qquad \textrm{for almost all $\alpha$},
\end{equation}
where $\sigma_{f_1}$ and $\sigma_{f_2}$ are defined according to~\eqref{sigmap}. Calculating the values of $\sigma_{f_1}$ and $\sigma_{f_2}$ is a simple exercise, using the Fourier series expansion of $f_1$ and $f_2$, respectively; it turns out that in our specific setting we have
\begin{equation} \label{LILf3}
\sigma_{f_1}^2 = \sum_{j=1}^{\infty} \left(\frac{1}{j^2} + 2 \sum_{r=1}^\infty \frac{1}{2^r j^2}\right) = \frac{\pi^2}{2}
\end{equation}
and
\begin{equation} \label{LILf4}
\sigma_{f_2}^2 = \sum_{j=1}^{\infty} \left(\frac{1}{j^2} + 2 \sum_{r=1}^\infty \underbrace{(-1)^j (-1)^{2^r j}}_{=(-1)^j} \frac{1}{2^r j^2}\right) = 0.
\end{equation} 
By applying the same arguments to $-f_1$ and $-f_2$ instead of $f_1$ and $f_2$ we can get absolute values in~\eqref{LILf1} and~\eqref{LILf2}, if we wish. This proves Lemma~\ref{lemmath1}.\\
\end{proof}

\begin{proof}[Proofs of Theorem~\ref{th_thue} and Theorem~\ref{th_theorem2}]~\\
\emph{Part 1: upper bounds.} \quad By periodicity it is obviously sufficient to prove Theorem~\ref{th_thue} for $h=1$. Let $\ve>0$ and $\alpha$ be given, and set $\hat{\ve} = \ve/2$. We will assume that $\alpha$ is an element of the set of full measure for which the conclusion of Lemma~\ref{lemmath1} holds. Then we have
\begin{equation} \label{th1f1}
\left| \sum_{\ell=0}^{L} f_1(2^\ell \alpha) \right| \leq \left(\frac{\pi}{\sqrt{2}}+\hat{\ve} \right)\sqrt{2 L \log \log L}
\end{equation}
and
\begin{equation} \label{th1f2}
\left| \sum_{\ell=0}^{L} f_2(2^\ell \alpha) \right| \leq \hat{\ve} \sqrt{2 L \log \log L}
\end{equation}
for all $L \geq L_0(\alpha)$. Equation~\eqref{th1f1} already gives the upper bound in Theorem~\ref{th_theorem2}. To obtain the upper bound in Theorem~\ref{th_thue}, let $N$ be given, and assume that $N \geq 2^{(2^{L_0})}$. We can write 
$$
N = \eta_M 2^M + \eta_{M-1} 2^{M-1} + \dots + \eta_1 2 + \eta_0
$$
for numbers $(\eta_M, \dots, \eta_0) \in \{0,1\}^{M+1}$, where we assume that $M$ is chosen in such a way that $\eta_M=1$; this is simply the binary representation of $N$. For simplicity of writing we set $N_{M+1}=0$ and 
$$
N_\mu = \eta^M 2^M + \eta_{M-1} 2^{M-1} + \dots + \eta_{\mu} 2^{\mu}, \qquad 0 \leq \mu \leq M.
$$
Then clearly we have
\begin{eqnarray} \label{sumN}
\sum_{k=1}^N e^{2 \pi i n_k \alpha} = \sum_{\mu=1}^{M+1} ~\eta_{\mu-1} \sum_{k= N_{\mu}+ 1}^{N_{\mu-1}} e^{2 \pi i n_k \alpha} = \sum_{\mu=1}^{M+1} ~\eta_{\mu-1} \sum_{k= 1}^{2^{\mu-1}} e^{2 \pi i n_{N_\mu+k} \alpha}.
\end{eqnarray}
Note that for the ``odious numbers'' $m_k$ we have, for every $k$, that
$$
m_k = \left\{ \begin{array}{ll} 2k-1 & \textrm{if $n_k = 2k-2$},\\ 2k-2 & \textrm{if $n_k=2k-1$.} \end{array} \right.
$$
Furthermore, from the special structure of the Thue--Morse sequence we see that for $1 \leq k \leq 2^{\mu-1}$ we have
\begin{equation} \label{nnmu}
n_{N_{\mu}+k} = \left\{ \begin{array}{ll} 2 N_{\mu} + n_k & \textrm{if $s_2(N_\mu)=0$},\\ 2 N_{\mu} + m_k & \textrm{if $s_2(N_\mu)=1,$}\end{array}\right.
\end{equation}
which together with~\eqref{lacproducts},~\eqref{lacproducts2},~\eqref{prod_lac_1} and~\eqref{prod_lac_2} implies that
$$
\left|\sum_{k= 1}^{2^{\mu-1}} e^{2 \pi i n_{N_\mu+k} \alpha}\right| \leq \left(\exp \left(\sum_{\ell=0}^{\mu-1} f_1(2^\ell \alpha) \right) + \exp \left(\sum_{\ell=0}^{\mu-1} f_2(2^\ell \alpha) \right) \right), \qquad 1 \leq \mu \leq M+1.
$$
Thus by~\eqref{sumN} and~\eqref{nnmu} we have
\begin{eqnarray}
\left|\sum_{k=1}^N e^{2 \pi i n_k \alpha}\right| & \leq & \sum_{\mu=1}^{M+1}  \left|\sum_{k= 1}^{2^{\mu-1}} e^{2 \pi i n_{N_\mu+k} \alpha}\right| \label{continue}\\
& \leq & 4 M + \sum_{\mu=\lceil \log_2 M \rceil}^{M+1}  \left|\sum_{k= 1}^{2^{\mu-1}} e^{2 \pi i n_{N_\mu+k} \alpha}\right| \nonumber\\
& \leq & 4 M + \sum_{\mu=\lceil \log_2 M \rceil}^{M+1}  \left(\exp \left(\sum_{\ell=0}^{\mu-1} f_1(2^\ell \alpha) \right) + \exp \left(\sum_{\ell=0}^{\mu-1} f_2(2^\ell \alpha) \right) \right). \label{co1}
\end{eqnarray}
Using the fact that the assumption $N \geq 2^{(2^{L_0})}$ implies that $\log_2 M \geq L_0$, and also using the inequalities~\eqref{th1f1} and~\eqref{th1f2}, we obtain
\begin{eqnarray*}
\sum_{\mu=\lceil \log_2 M \rceil}^{M+1}  \exp \left(\sum_{\ell=0}^{\mu-1} f_1(2^\ell \alpha) \right) & \leq & \sum_{\mu=\lceil \log_2 M \rceil}^{M+1}  \exp \left(\left(\frac{\pi}{\sqrt{2}}+\hat{\ve} \right)\sqrt{2 (\mu-1) \log \log (\mu-1)} \right) \\
& \leq & M \exp \left(\left(\frac{\pi}{\sqrt{2}}+\hat{\ve} \right)\sqrt{2 M \log \log M} \right) \label{co2}
\end{eqnarray*}
and
\begin{eqnarray*}
\sum_{\mu=\lceil \log_2 M \rceil}^{M+1}  \exp \left(\sum_{\ell=0}^{\mu-1} f_2(2^\ell \alpha) \right) & \leq & \sum_{\mu=\lceil \log_2 M \rceil}^{M+1}   \exp \left(\hat{\ve} \sqrt{2 (\mu-1) \log \log (\mu-1)} \right) \\
& \leq & M \exp \left(\hat{\ve} \sqrt{2 M \log \log M} \right). \label{co3}
\end{eqnarray*}
Combining~\eqref{co1},~\eqref{co2} and~\eqref{co3} we obtain 
\begin{eqnarray*}
\left|\sum_{k=1}^N e^{2 \pi i n_k \alpha}\right| & \leq & 4M + 2M \exp \left(\left(\frac{\pi}{\sqrt{2}}+\hat{\ve} \right)\sqrt{2 M \log \log M} \right) \\
& \leq & 4M + 2M \exp \left(\left(\frac{\pi}{\sqrt{2}}+\hat{\ve} \right)\sqrt{2 (\log_2 N) \log \log (\log_2 N)} \right).
\end{eqnarray*}
As a consequence we have
$$
\left|\sum_{k=1}^N e^{2 \pi i n_k \alpha}\right| \leq \exp \left(\left(\frac{\pi}{\sqrt{\log 2}} +\ve \right)\sqrt{(\log N) \log \log \log N} \right)
$$
for all sufficiently large $N$. This proves the upper bound in Theorem~\ref{th_thue}.\\

\emph{Part 2: lower bounds.} \quad Now we prove the lower bound in Theorems~\ref{th_thue} and~\ref{th_theorem2}. Again we assume that $h=1$, that $\alpha$ and $\ve>0$ are fixed, and that $\alpha$ is from the set of full measure for which the conclusion of Lemma~\ref{lemmath1} holds. Again we set $\hat{\ve} = \ve/2$. Then by Lemma~\ref{lemmath1} there exist infinitely many values of $L$ for which both inequalities
$$
\sum_{\ell=0}^{L} f_1(2^\ell \alpha) \geq \left(\frac{\pi}{\sqrt{2}} - \hat{\ve}\right) \sqrt{2 L \log \log L}
$$
and
$$
\sum_{\ell=0}^{L} f_2(2^\ell \alpha) \leq \hat{\ve} \sqrt{2 L \log \log L}
$$
hold simultaneously. The first of the two relations already gives the lower bound in Theorem~\ref{th_theorem2}. By~\eqref{lacproducts} and~\eqref{prod_lac_1},~\eqref{prod_lac_2} we have
\begin{equation*} 
\left| \sum_{k=1}^{2^L} e^{2 \pi i n_k \alpha} \right| \geq \frac{1}{2} \exp \left(\sum_{\ell=0}^{L} f_1(2^\ell \alpha) \right) - \frac{1}{2} \exp\left(\sum_{\ell=0}^{L} f_2(2^\ell \alpha) \right).
\end{equation*}
Thus for infinitely many $L$ we have
\begin{equation} 
\left|\sum_{k=1}^{2^L} e^{2 \pi i n_k \alpha} \right| \geq \frac{1}{2} \exp \left(\left(\frac{\pi}{\sqrt{2}} - \hat{\ve}\right) \sqrt{2 L \log \log L}\right) - \frac{1}{2} \exp \left(\hat{\ve} \sqrt{2 L \log \log L}\right).
\end{equation}
Consequently we also have
$$
\left|\sum_{k=1}^{N} e^{2 \pi i n_k \alpha} \right| \geq \exp \left(\left(\frac{\pi}{\sqrt{\log 2}} - \ve \right) \sqrt{(\log N) \log \log \log N}\right)
$$
for infinitely many $N$. This proves the lower bound in Theorem~\ref{th_thue}.\\
\end{proof}

\section{Discrepancy of Thue--Morse--Kronecker sequences: Proof of Theorem~\ref{th_thue2}} \label{sect_proof_th2}

For given $L \geq 1$, we set
$$
I_1 (L) = \int_0^1 \left( \prod_{\ell=0}^{L-1} \left|2 \sin (\pi 2^\ell \alpha)\right| \right)~d\alpha
$$
and
$$
I_2 (L) = \int_0^1 \left( \prod_{\ell=0}^{L-1} \left|2 \cos(\pi 2^\ell \alpha)\right| \right)~d\alpha.
$$
Integrals of this type have been studied in great detail in~\cite{foumau}. For the integral $I_1$ it is proved there that
\begin{equation} \label{0664}
2^{-L} I_1(L) = \kappa \lambda^{L} \left(1+o(1)\right)
\end{equation}
where $\kappa,\lambda$ are positive constants with $0.654336 \leq \lambda \leq 0.663197$ (thereby improving an earlier result of {\`E}minyan~\cite{emin}).\\
Hence for every $\ve > 0$ for L large enough we have
\begin{equation} \label{0664b}
I_1 (L) \leq \left(2^{L}\right)^{1+\log_{2} \lambda + \varepsilon}.
\end{equation}

We will improve the estimate given for $\lambda$ by Fouvry and Mauduit in the following.
\begin{lemma} \label{lem7}
Let $\lambda$ be defined as in~\eqref{0664}. Then $0.66130 < \lambda < 0.66135.$
\end{lemma}
\begin{proof}
By the formula above of equation $(4.2)$ in~\cite{foumau} we have 
$$
\int^{1}_{0} \prod^{L-1}_{\ell=0} \left|\sin \pi 2^{\ell} \alpha \right|d\alpha = \int^{1}_{0} \phi_{j} (\alpha) \prod^{L-j-1}_{\ell=0} \left|\sin \pi 2^{\ell} \alpha\right| d\alpha
$$
where $\phi_{0} (\alpha) \equiv 1$ and
$$
\phi_{j+1} (\alpha) = \frac{1}{2} \left(\left|\sin \pi \frac{\alpha}{2}\right| \phi_{j} \left(\frac{\alpha}{2}\right) + \left|\cos \pi \frac{\alpha}{2}\right| \phi_{j} \left(\frac{\alpha+1}{2}\right)\right).
$$
Furthermore, it was shown in~\cite{foumau} that the functions $\phi_{j}$ are symmetric around $\alpha= \frac{1}{2}$ on $\left[0,1\right]$, and that they are concave on $\left[0,1\right]$. Hence $\phi_{j}(0)= \underset{\alpha \in \left[0,1\right]}{\min} \phi_{j} (\alpha)$ and
$$
\phi_{j} \left(\frac{1}{2}\right) = \underset{\alpha \in \left[0,1\right]}{\max} \phi_{j} (\alpha).
$$
Let $q_{j} (\alpha) := \frac{\phi_{j} (\alpha)}{\phi_{j-1} (\alpha)}$ and $m_{j} := \underset{\alpha \in \left[0,1\right]}{\min} q_{j} (\alpha), M_{j} := \underset{\alpha \in \left[0,1\right]}{\max} q_{j} (\alpha).$\\
Note that $q_{j} (\alpha)$ of course also is symmetric around $\alpha=\frac{1}{2}$ in $\left[0,1\right]$. We have for every $\alpha in [0,1]$
\begin{eqnarray*}
q_{j+1} (\alpha) & = & \frac{\phi_{j+1} (\alpha)}{\phi_{j} (\alpha)} \\
& = & \frac{\left|\sin \pi \frac{\alpha}{2}\right| \phi_{j} \left(\frac{\alpha}{2}\right) + \left|\cos \pi \frac{\alpha}{2}\right| \phi_{j} \left(\frac{\alpha+1}{2}\right)}{\left|\sin \pi \frac{\alpha}{2}\right| \phi_{j-1} \left(\frac{\alpha}{2}\right) + \left|\cos \pi \frac{\alpha}{2}\right| \phi_{j-1} \left(\frac{\alpha+1}{2}\right)}  \\
& \leq & \frac{\left|\sin \pi \frac{\alpha}{2}\right|M_{j} \phi_{j-1} \left(\frac{\alpha}{2}\right)+ \left|\cos \pi \frac{\alpha}{2}\right| M_{j} \phi_{j-1} \left(\frac{\alpha+1}{2}\right)}{\left|\sin \pi \frac{\alpha}{2}\right| \phi_{j-1} \left(\frac{\alpha}{2}\right) + \left|\cos \pi \frac{\alpha}{2}\right| \phi_{j-1} \left(\frac{\alpha+1}{2}\right)} \\
& = & M_{j}.
\end{eqnarray*}
Therefore $M_{j+1}=\underset{\alpha \in \left[0,1\right]}{\max} q_{j} (\alpha)\leq M_j$. \\
Analogously we obtain $q_{j+1} (\alpha) \geq m_{j}$ for all $\alpha \in \left[0,1\right]$. Altogether $M_{1} \geq M_{2} \geq M_{3} \geq \ldots$ and $m_{1} \leq m_{2} \leq m_{3} \leq \ldots$, and therefore for every $k$ fixed we have
\begin{eqnarray*}
\int^{1}_{0} \prod^{L-1}_{\ell=0} \left|\sin \pi 2^{\ell} \alpha\right| d\alpha & = & \int^{1}_{0} \phi_{L} (\alpha) d\alpha \\
& = & \int^{1}_{0} \prod^{L}_{j=1} q_{j} (\alpha) d\alpha \\ 
& \leq & M_{k}^{L-k} \int^{1}_{0} q_{1} (\alpha) \ldots q_{k-1} (\alpha) ~d\alpha.
\end{eqnarray*}
Similarly we get
$$
\int^{1}_{0} \prod^{L-1}_{\ell=0} \left|\sin \pi 2^{\ell} \alpha\right|d\alpha \geq m_{k}^{L-k} \int^{1}_{0} q_{1} (\alpha) \ldots q_{k-1} (\alpha) ~d\alpha.
$$
Hence
\begin{equation*} 
m_{k}^{L-k} \int^{1}_{0} q_{1} (\alpha) \ldots q_{k-1} (\alpha) ~d\alpha \leq \kappa \lambda^{L} \left(1+o(1)\right) \leq M_{k}^{L-k} \int^{1}_{0} q_{1} (\alpha) \ldots q_{k-1} (\alpha) ~d\alpha 
\end{equation*}
and consequently 
\begin{equation}\label{mkmk}
m_{k} \leq \lambda \leq M_{k} \qquad \textrm{for all $k \geq 1$.} 
\end{equation}
By considering the function $q_6(\alpha)$, in the following we will prove that $m_{6} > 0.6613$ and $M_{6} < 0.66135$.\\
\begin{center}
\includegraphics[angle=0,width=70mm]{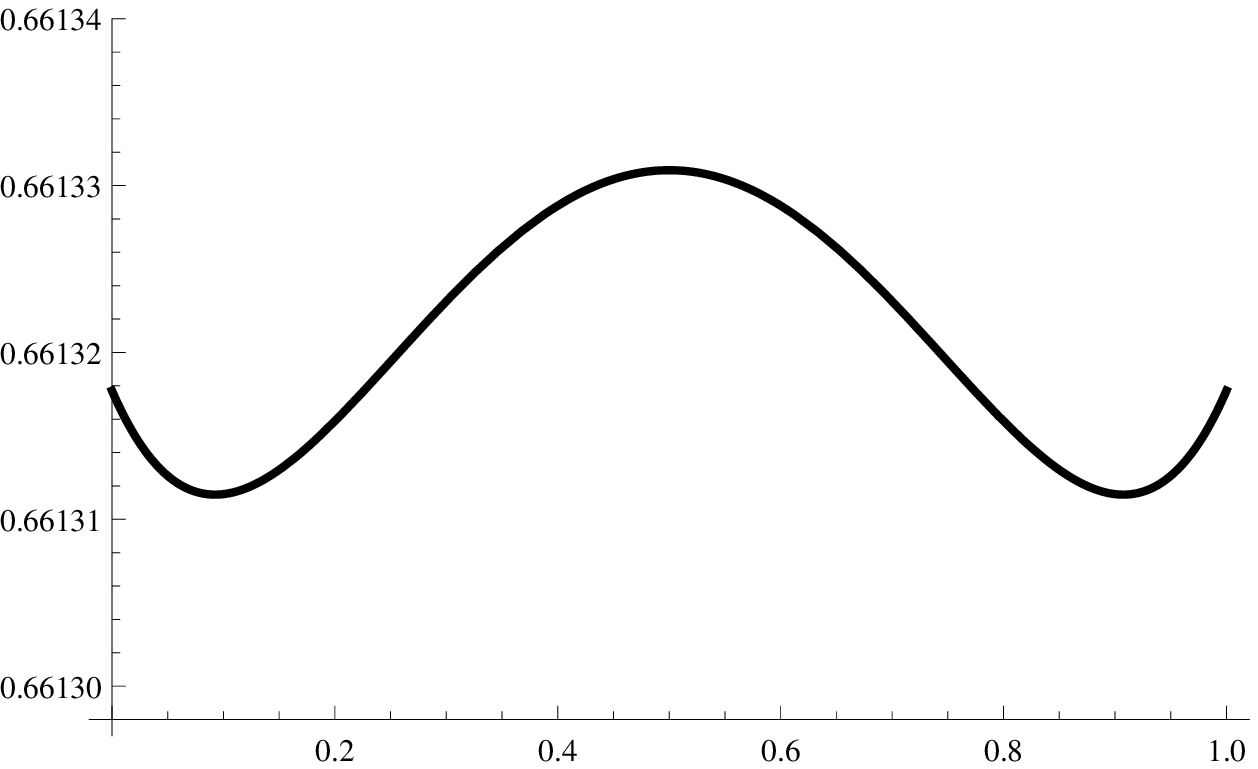}\\
Figure 2: the function $q_6(\alpha)$.
\end{center}

Wherever $q_{6}$ is differentiable we have
\begin{eqnarray}
\left|{q_{6}}' (\alpha)\right| & = & \left|\frac{{\phi_{6}}' (\alpha) \phi_{5} (\alpha) - \phi_{6} (\alpha) {\phi_{5}}' (\alpha)}{\left(\phi_{5}(\alpha)\right)^{2}}\right| \nonumber\\
& \leq & \frac{\underset{\alpha \in \left[0,1\right]}{\max} \left|{\phi_{6}}' (\alpha)\right|}{\phi_{5}(0)} + \frac{\phi_{6} \left(\frac{1}{2}\right)}{\left(\phi_{5} (0)\right)^{2}} \underset{\alpha \in \left[0,1\right]}{\max} \left|{\phi_{5}}'(\alpha)\right|. \label{phiequ}
\end{eqnarray}
It can easily be checked for example by differentiating $\phi_{5} (\alpha)$ and $\phi_{6}(\alpha)$ with the help of Mathematica that ${\phi_{5}}' (\alpha)$ is the sum resp. difference of\\
32 products of absolute values of sines and cosines, each product weighted by a factor $\frac{\pi}{1024}$ and further\\
32 such products weighted by $\frac{\pi}{512}$, further\\
32 with weight $\frac{\pi}{256}$, further\\
32 with weight $\frac{\pi}{128}$ and finally further\\
32 products with factor $\frac{\pi}{64}$.\\
Hence
$$\left|{\phi_{5}}' (\alpha)\right| \leq 32 \pi \left(\frac{1}{1024}+ \frac{1}{512} + \frac{1}{256} + \frac{1}{128} + \frac{1}{64}\right)=$$
$$=\frac{31}{32} \pi.$$
In the same way we show that also $\left|{\phi_{6}}'(\alpha)\right|\leq \frac{31}{32} \pi$. By combining these estimates with the values of $\phi_5(0)$ and $\phi_6(1/2)$ in~\eqref{phiequ}, we finally obtain  $\left|{q_{6}}' (\alpha)\right|\leq 56.4$.\\

Now we calculate
$$
q_{6} \left(\frac{a}{2800000}\right) \mbox{for}~ a=0,1, \ldots, 1400000
$$
with the help of Mathematica and obtain
$$
\max\left\{q_{6} \left(\frac{a}{2800000}\right) \bigg\vert a=0,1, \ldots, 1400000\right\} = 0.66133092 \ldots
$$
$$
\min\left\{q_{6} \left(\frac{a}{2800000}\right) \bigg\vert a=0,1, \ldots, 1400000\right\}= 0.66131148 \ldots
$$
Hence
$$
m_{6} \geq 0.66131145 - 56.4 \frac{1}{5600000} = 0.661301 \ldots
$$
and
$$
M_{6} \leq 0.66133092 + 56.4 \frac{1}{5600000} = 0.661341 \ldots
$$
By~\eqref{mkmk} this implies Lemma~\ref{lem7}.
\end{proof}
Let us remark that numerical experiments with $q_{15} (\alpha)$ suggest that $\lambda=0.661322602 \ldots.$ It is tempting to conjecture that the precise value of $\lambda$ can be expressed in a simple way in terms of the ``usual'' mathematical constants such as $e,\pi, \log 2$, etc. However, we do not know what such an expression could look like, and cannot even make a reasonable guess (the numerical argument in the proof of Lemma~\ref{lem7} does not give any hints).\\

\begin{proof}[Proof of Theorem~\ref{th_thue2}]~\\
\emph{Part 1: upper bound}\\

For the integral $I_2 (L)$ we can use the equality
\begin{equation} \label{BL3}
\prod_{\ell=0}^{L-1} \left| 2 \cos (\pi 2^\ell \alpha) \right| = \frac{\left| \sin \pi 2^{L} \alpha \right|}{\sin \pi \alpha} \leq \min \left(2^L, \frac{1}{\pi \alpha (1-\alpha)}\right)
\end{equation}
which holds for $0<\alpha<1$ and which implies that
\begin{eqnarray}
I_2 (L) & \leq & \int_0^1 \frac{\left| \sin \pi 2^{L} \alpha \right|}{\sin \pi \alpha} ~d\alpha \nonumber\\
& \leq & 2 \int_0^{2^{-L}} 2^L ~d\alpha + \int_{2^{-L}}^{1-2^{-L}} \frac{1}{\pi \alpha (1-\alpha)} ~d\alpha \nonumber\\
& \leq & 2 + \frac{L 2 \log 2}{\pi} \label{0664c}
\end{eqnarray}
(this is essentially a variant of the classical bound for the $L^{1}$-norm of the Dirichlet kernel).\\
For $\mu\geq 1$ and $\ve > 0$ we set

\begin{equation} \label{gmu}
G_\mu = \left( \alpha \in (0,1):~\sum_{h=1}^{2^{4\mu}} \frac{1}{h} \left| \sum_{k=1}^{2^{\mu-1}} e^{2 \pi i h n_k \alpha} \right| > (2^\mu)^{\lambda+\ve} \right).
\end{equation}
By~\eqref{lacproducts},~\eqref{0664b} and~\eqref{0664c} we have
\begin{eqnarray*}
\int_0^1 \sum_{h=1}^{2^{4\mu}} \frac{1}{h} \left| \sum_{k=1}^{2^{\mu-1}} e^{2 \pi i h n_k \alpha} \right| ~d\alpha & = & \sum_{h=1}^{2^{4\mu}} \frac{1}{h} \int_0^1 \left| \sum_{k=1}^{2^{\mu-1}}
e^{2 \pi i h n_k \alpha} \right| ~d\alpha \\
& \leq & \sum_{h=1}^{2^{4\mu}} \frac{1}{h} \left(I_1(\mu) + I_2(\mu)\right) \\
& \leq & (2^\mu)^{\lambda+\frac{\ve}{2}}
\end{eqnarray*}
for sufficiently large $\mu$. Consequently we have
$$
\p  (G_\mu) \leq (2^\mu)^{-\frac{\ve}{2}}.
$$
which implies that by the Borel--Cantelli lemma with probability one only finitely many events $G_\mu$ occur. We can show the same result if we replace the sequence $(n_k)_{k \geq 1}$ in~\eqref{gmu} by $(m_k)_{k \geq 1}$. In other words, for almost all $\alpha$ we have 
\begin{equation} \label{4095}
\sum_{h=1}^{2^{4\mu}} \frac{1}{h} \left| \sum_{k=1}^{2^{\mu-1}} e^{2 \pi i h n_k \alpha} \right| \leq (2^\mu)^{\lambda+\ve}
\end{equation}
and the same estimate for $(m_k)_{k \geq 1}$ instead of $(n_k)_{k \geq 1}$, for all $\mu \geq \mu_0 ( \alpha,\ve)$.\\

Now assume that $\alpha, \ve$ and $N$ are given. Furthermore we assume that for these values of $\alpha$ and $\ve$ the estimate~\eqref{4095} and the corresponding estimate for $(m_k)_{k \geq 1}$ instead of $(n_k)_{k \geq 1}$ hold for $\mu \geq \mu_0$, and that $N \geq 2^{(4\mu_0)}$. We apply the same dyadic decomposition of $N$ as in the proof of the upper bound of Theorem~\ref{th_thue} in Section~\ref{sect_proof}. In the same way as we obtained~\eqref{continue}, together with the Koksma-Erd\H{o}s-Tur\'an inequality we can now obtain with $M=\lfloor \log_2(N)\rfloor$
\begin{eqnarray}
& & N D_N^* (\{n_1 \alpha\}, \dots, \{n_N \alpha\}) \nonumber\\
& \leq & 1 + \sum_{h=1}^N \frac{1}{h} \left|\sum_{k=1}^N e^{2 \pi i h n_k \alpha}\right| \nonumber\\
& \leq & 1 + \sum_{h=1}^N \frac{1}{h} \sum_{\mu=1}^{M+1}  \left|\sum_{k= 1}^{2^{\mu-1}} e^{2 \pi i h n_{N_\mu+k} \alpha}\right| \nonumber\\
& \leq & 1 + \underbrace{ \sum_{\mu=1}^{\lceil (\log_2 N)/4\rceil} \sum_{h=1}^N \frac{1}{h} \left|\sum_{k= 1}^{2^{\mu-1}} e^{2 \pi i h n_{N_\mu+k} \alpha}\right|}_{\ll N^{1/4} \log N}  + \sum_{\mu=\lceil (\log_2 N)/4\rceil+1}^{M+1} \sum_{h=1}^N \frac{1}{h}  \left|\sum_{k= 1}^{2^{\mu-1}} e^{2 \pi i h n_{N_\mu+k} \alpha}\right|. \label{lastt}
\end{eqnarray}
For the last term in~\eqref{lastt} by~\eqref{nnmu} and~\eqref{4095} we have
\begin{eqnarray*}
\sum_{\mu=\lceil (\log_2 N)/4\rceil+1}^{M+1} \sum_{h=1}^N \frac{1}{h}  \left|\sum_{k= 1}^{2^{\mu-1}} e^{2 \pi i h n_{N_\mu+k} \alpha}\right| & \leq & \sum_{\mu=\lceil (\log_2 N)/4\rceil+1}^{M+1} ~\sum_{h=1}^{2^{4\mu}} \frac{1}{h}  \left|\sum_{k= 1}^{2^{\mu-1}} e^{2 \pi i h n_{N_\mu+k} \alpha}\right| \\
& \leq & \underbrace{\sum_{\mu=\lceil (\log_2 N)/4\rceil+1}^{M+1} (2^\mu)^{\lambda+\ve/2}}_{\ll (2^M)^{\lambda+\ve/2}}.
\end{eqnarray*}
Since $M \leq \log_2 N$, together with~\eqref{lastt} we have shown that 
$$
N D_N^* (\{n_1 \alpha\}, \dots, \{n_N \alpha\}) \leq N^{\lambda+\ve} 
$$
for all sufficiently large $N$, which proves the upper bound in Theorem~\ref{th_thue2}.\\

\emph{Part 2: lower bound.}\\

By~\eqref{BL1} and~\eqref{lacproductsLB} for the discrepancy $D^{*}_{N}$ of the sequence $(\left\{n_{k} \alpha\right\})_{k \geq 1}$ with $N=2^{L}$ for each positive integer $H$ we have 
\begin{eqnarray*}
D^{*}_{N} & \geq & \frac{1}{4H} \left|\frac{1}{N} \sum^{N}_{k=1} e^{2 \pi i n_{k} H \alpha }\right| \\
& \geq & \frac{1}{4H} \prod^{L}_{\ell=0} \left|\sin \pi H 2^{\ell} \alpha \right|- \underbrace{\frac{1}{4H} \prod^{L}_{\ell=0} \left|\cos \pi H 2^{\ell} \alpha \right|}_{\ll(H \left\|H \alpha \right\|)^{-1}}.
\end{eqnarray*}
Let
$$
f_L (\alpha) = \prod^{L}_{\ell=0} \left|\sin \pi 2^{\ell} \alpha \right|.
$$
We will show below that for any given $\varepsilon > 0$ for almost all $\alpha$ there are infinitely many $L$ such that there exists a positive integer $h_{L}$ with $h_{L} \leq 2^{L}$ and
\begin{equation} \label{F2}
\frac{1}{h_{L}} f_{L} \left(h_L \alpha \right) \gg \lambda^{L(1+\ve)}.
\end{equation}
It is a well-known fact in metric Diophantine approximation that for almost all $\alpha$ we have $h \left\|h \alpha \right\| \geq \frac{1}{h^{\ve}}$ for all $h$ large enough. Hence if~\eqref{F2} is true for almost all $\alpha$ then there are infinitely many $L$ such that for $N=2^{L}$ we have
\begin{eqnarray}
D^{*}_{N} & \gg & \lambda^{L(1+\ve)} - \left(h_{L}\right)^{\varepsilon} \label{HL}\\
 & \gg & \left(2^{L}\right)^{\left(1+ \varepsilon\right) \frac{\log \lambda}{\log 2}} - \left(2^{L}\right)^{\varepsilon}\nonumber\\
& \gg & N^{\left(1+ \varepsilon\right) \frac{\log \lambda}{\log 2}}\nonumber,
\end{eqnarray}
and the desired result follows (note that $\log \lambda <0$). It remains to show the existence of the numbers $h_{L} \leq 2^{L}$ which satisfy~\eqref{F2}.\\

Let $\ve > 0$ be given. From the definition of $f_L(\alpha)$ it is easily seen that
\begin{equation} \label{derivative}
\left|f_L(\alpha_1) - f_L(\alpha_2)\right| \leq 2^{L+1} \pi |\alpha_1 - \alpha_2|;
\end{equation}
this follows from the fact that the derivative of the function $\prod^{L}_{\ell=0} \sin \pi 2^{\ell} \alpha$ is bounded uniformly by $2^{L+1} \pi$. Now let $g_L(\alpha)$ be the function defined by
$$
g_L(\alpha) = f_L(j 4^{-L}) \qquad \textrm{for $\alpha \in \Big[j 4^{-L}, (j+1) 4^{-L}\Big)$}, \qquad \textrm{for $j=0, \dots, 4^L - 1$}. 
$$
This definition means that $g_L$ is constant on intervals of length $4^{-L}$ which lie between two integer multiples of $4^{-L}$, and coincides with $f_L$ on the left endpoint of such intervals. By~\eqref{derivative} we have
\begin{equation} \label{gf}
\left|g_L - f_L\right| \leq 2 \pi 2^{-L},
\end{equation}
which means that it is sufficient to prove~\eqref{F2} with $f_L$ replaced by $g_L$ (remember that the value of $\ve>0$ was arbitrary and $\lambda>1/2$). The reason for using the functions $g_L$ instead of $f_L$ is that every function $g_L$ can be written as a sum of at most $4^L$ different indicator functions of intervals; consequently, we know that the set of values of $\alpha$ where $|g_L|$ is ``large'' can be written as the union of at most $4^L$ intervals, which implies an upper bound for the size of the Fourier coefficients of the indicator function of this set (see below for details).\\

Let $Q=Q\left(\ve\right)$ be a positive integer which will be chosen in dependence on $\ve$ (we assume that $Q$ is ``large''). We define real numbers
$$
\delta_{i} = \left(\frac{1}{2}\right)^{1-\frac{i}{Q+1}}, \qquad i = 0, 1, \dots, Q+1.
$$
Furthermore, we define 
$$
M^{\left(i\right)}_{L} := \Big\{\alpha \in \left[0,1\right):~ \delta_{i}^{L} < |g_{L} \left(\alpha \right)| \leq \delta^{L}_{i+1}\Big\}
$$ 
for $i=0,1,\ldots, Q$.\\ 

Then by~\eqref{0664} and~\eqref{gf} we have
\begin{eqnarray*}
& & \sum^{Q}_{i=0} \mathbb{P} \left(M^{\left(i\right)}_{L}\right) \delta^{L}_{i+1} + \left(1-\sum^{Q}_{i=0} \mathbb{P} \left(M^{\left(i\right)}_{L}\right)\right) 2^{-L}\\
& \geq & \int^{1}_{0} |g_{L}\left(\alpha \right)| ~d\alpha  \\
& > & \frac{\kappa}{2} \lambda^{L}
\end{eqnarray*}
for sufficiently large $L$, where $\kappa$ and $\lambda$ are the numbers from~\eqref{0664}, and where we used the fact that $\lambda > 1/2$. Hence we have
$$
\sum^{Q}_{i=0} \mathbb{P} \left(M^{\left(i\right)}_{L}\right) \delta^{L}_{i+1} \geq \frac{\kappa}{4} \lambda^{L}
$$
for sufficiently large $L$. Consequently for every $L$ large enough there is an $i_{L} \in \left\{0, \ldots, Q\right\}$ with
$$
\delta_{i_{L}+1}^L \mathbb{P} \left(M^{\left(i_{L}\right)}_{L}\right) \geq \frac{\kappa}{4 Q} \lambda^{L},
$$
which implies that
$$
\mathbb{P}\left(M^{\left(i_{L}\right)}_{L}\right) \geq \frac{\kappa}{4 Q} \left(\frac{\lambda}{\delta_{i_{L}+1}}\right)^{L}.
$$
Note that, as a consequence of the construction of $g_L$, the set $M^{\left(i\right)}_{L}$ always is a union of at most $4^L$ disjoint intervals. It is easily seen that by trimming the sets $M^{\left(i\right)}_{L}$ appropriately we can always find a set $R^{\left(i\right)}_{L}$ such that $R^{\left(i\right)}_{L} \subset M^{\left(i\right)}_{L}$, such that $R^{\left(i\right)}_{L}$ also is the union of at most $4^L$ intervals, and such that for the measure of the sets $R^{\left(i\right)}_{L}$ we have the exact equality
$$
\mathbb{P} \left(R_{L}^{(i_{L})}\right) = \frac{\kappa}{4 Q} \left(\frac{\lambda}{\delta_{i_{L}+1}}\right)^{L}.
$$

Let $\eta=\eta(\ve)>0$ be a ``small'' number. Let $H_L$ denote the largest integer such that
\begin{equation} \label{hldef}
H_L \leq \frac{4 Q}{\kappa} \left(\frac{\delta_{i_L+1}}{\lambda}\right)^{L} (1+\eta)^L.
\end{equation}
Note that the right-hand side of~\eqref{hldef} can be written as
$$
\frac{1}{\mathbb{P} \left(R_{L}^{(i_{L})}\right)}~  (1+\eta)^L,
$$
and consequently we have $H_L \geq (1+\eta)^L$; that means, $H_L$ grows exponentially in $L$. Note also that it is easily seen that $H_L\leq 2^L$ for sufficiently large $L$ (provided that $\eta$ is chosen sufficiently small), which is important for~\eqref{HL}.\\

We will show that for almost all $\alpha$ for infinitely many $L$ there is a
\begin{equation} \label{F10}
h \leq H_L~\mbox{such that}~\left\{h \alpha\right\} \in R_{L}^{(i_L)}.
\end{equation}
For these $h$ then we have
\begin{eqnarray*}
\frac{|g_{L}(h \alpha)|}{h} & \geq & \delta_{i_L}^{L} \frac{\kappa}{4 Q} \left(\frac{\lambda}{\delta_{i_L+1}}\right)^{L} (1+\eta)^{-L} \\
& = & \frac{\kappa}{4 Q} \left(\frac{\delta_{i_L}}{\delta_{i_L+1}}\right)^{L}  (1+\eta)^{-L} \lambda^{L} \\
& \gg & \lambda^{L(1+\ve)}, \\
\end{eqnarray*} 
for $Q$ large enough and $\eta$ small enough in dependence on $\ve$. Together with~\eqref{gf} this will establish~\eqref{F2}, as desired.\\

It remains to show~\eqref{F10}. Let $\mathbf{1}_L(\alpha)$ denote the indicator function of the set $R^{\left(i\right)}_{L}$, extended with period one. Then we know that
\begin{equation} \label{integ}
\int_0^1 \mathbf{1}_L (\alpha) ~d \alpha =  \frac{\kappa}{4 Q} \left(\frac{\lambda}{\delta_{i_{L}+1}}\right)^{L}.
\end{equation}
Setting
$$
\mathbb{I}_L (\alpha) = \mathbf{1}_L (\alpha) - \int_0^1 \mathbf{1}_L (\omega) ~d \omega, 
$$
we clearly have $\int_0^1 \mathbb{I}_L(\alpha) ~d\alpha=0$ and 
\begin{equation} \label{ilvar}
\textup{Var}_{[0,1]} ~ \mathbb{I}_L \leq 4^L. 
\end{equation}
From~\eqref{integ} we can easily calculate that
\begin{equation} \label{il2norm}
\|\mathbb{I}_L\|_2^2 = \int_0^1 \mathbb{I}_L (\alpha)^2  d\alpha = \frac{\kappa}{4 Q} \left(\frac{\lambda}{\delta_{i_{L}+1}}\right)^{L} \left(1 - \frac{\kappa}{4 Q} \left(\frac{\lambda}{\delta_{i_{L}+1}}\right)^{L} \right) ~\leq~ \frac{\kappa}{4 Q} \left(\frac{\lambda}{\delta_{i_{L}+1}}\right)^{L}.
\end{equation}
We write 
$$
\mathbb{I}_L (\alpha) \sim \sum_{j=1}^\infty \left(a_j \cos 2 \pi j \alpha + b_j \sin 2 \pi j \alpha\right)
$$
for the Fourier series of $\mathbb{I}_L$ (note that it has no constant term, since $\mathbb{I}_L$ has integral zero). In the sequel, we want to show that the sum 
\begin{equation} \label{ai1}
\sum_{h \leq H_L} \int_0^1 \mathbf{1}_L (h \alpha) ~d \alpha
\end{equation}
is large in comparison with the sum 
\begin{equation} \label{ai2}
\sum_{h\leq H_L} \mathbb{I}_L (h \alpha),
\end{equation}
for almost all $\alpha$ and infinitely many $L$. Since
\begin{equation} \label{lefthand}
\sum_{h\leq H_L} \mathbf{1}_L (h \alpha) = \sum_{h\leq H_L} \int_0^1 \mathbf{1}_L (\omega) ~d \omega + \sum_{h\leq H_L} \mathbb{I}_L (h \alpha),
\end{equation}
such an estimate will show that the sum on the left-hand side of~\eqref{lefthand} is large (for almost all $\alpha$, for infinitely many $L$), which in turn implies that many of the events described in~\eqref{F10} will occur. A lower bound for~\eqref{ai1} is easy to obtain; to find an asymptotic upper bound for~\eqref{ai2}, we will calculate the $L^2$ norm of these sums, and apply the Borel-Cantelli lemma.\\

From~\eqref{integ} we directly obtain
\begin{equation} \label{summe}
\sum_{h=1}^{H_L} \int_0^1 \mathbf{1}_L (\alpha) ~d \alpha = H_L \frac{\kappa}{4 Q} \left(\frac{\lambda}{\delta_{i_{L}+1}}\right)^{L} \gg (1+\eta)^L.
\end{equation}
Next we estimate 
$$
\left\| \sum_{h=1}^{H_L} \mathbb{I}_L (h \cdot) \right\|_2,
$$
which is relatively difficult. As a consequence of~\eqref{ilvar} and a classical inequality for the size of the Fourier coefficients of functions of bounded variation (see for example~\cite[p. 48]{zyg}) we have
\begin{equation} \label{furk}
|a_j| \leq \frac{\textup{Var}_{[0,1]} \mathbb{I}_L}{2j} \leq \frac{4^L}{j}, \qquad \textrm{and similarly} \qquad |b_j| \leq \frac{4^L}{j}.
\end{equation}
We split the function $\mathbb{I}_L$ into an even and an odd part (that is, into a cosine- and a sine-series). In the sequel, we consider only the even part; the odd part can be treated in exactly the same way. Let $p_L(\alpha)$ denote the $4^{3L}$-th partial sum of the Fourier series of the even part of $\mathbb{I}_L$, and let $r_L(\alpha)$ denote the remainder term. Then by Minkowski's inequality we have
\begin{equation} \label{split}
\left\| \sum_{h=1}^{H_L} \mathbb{I}_L^{(\textup{even})} (h \cdot) \right\|_2 \leq \left\| \sum_{h=1}^{H_L} p_L (h \cdot) \right\|_2 + \left\| \sum_{h=1}^{H_L} r_L (h \cdot) \right\|_2.
\end{equation}
Furthermore,~\eqref{furk}, Minkowski's inequality, and Parseval's identity imply that
\begin{eqnarray}
\left\|\sum_{h=1}^{H_L} r_L (h \cdot) \right\|_2 & \leq & H_L \|r_L\|_2 \nonumber\\
& \leq & H_L \sqrt{\sum_{j=4^{3L} + 1}^\infty \frac{4^{2L}}{j^2}} \nonumber\\
& \leq & H_L 2^{-L} \nonumber\\
& \ll & 1. \label{split2}
\end{eqnarray}
To estimate the first term on the right-hand side of~\eqref{split}, we expand $p_L$ into a Fourier series and use the orthogonality of the trigonometric system. Then we obtain
\begin{eqnarray}
\left\| \sum_{h=1}^{H_L} p_L (h \cdot) \right\|_2^2 & = & \underbrace{\sum_{n_1,n_2=1}^{H_L} \sum_{j_1,j_2=1}^{4^{3L}}}_{j_1 n_1 = j_2 n_2} \frac{a_{j_1} a_{j_2}}{2}  \nonumber\\
& = &  \sum_{j_1,j_2=1}^{4^{3L}} \frac{a_{j_1} a_{j_2}}{2} ~\# \Big\{ (n_1, n_2):~1 \leq n_1, n_2 \leq H_L, ~j_1 n_1 = j_2 n_2 \Big\}. \label{rhssum}
\end{eqnarray}
To estimate the size of the sum on the right-hand size of~\eqref{rhssum}, we assume that $j_1$ and $j_2$ are fixed. In the case $j_1=1$ and $j_2=1$, we clearly have $j_1 n_1 = j_2 n_2$ whenever $n_1 = n_2$; thus the cardinality of the set on the right-hand side of~\eqref{rhssum} is $H_L$. If $j_1=1$ and $j_2=2$, then we have to count the number of pairs $(n_1,n_2)$ for which $2 n_1 = n_2$; this number is $\lfloor H_L / 2 \rfloor$. For the values $j_1=2$ and $j_2=4$ we also have to count the number of pairs $(n_1,n_2)$ for which $2 n_1 = n_2$; so this cardinality is also $\lfloor H_L / 2 \rfloor$. The last example shows that the greatest common divisor of $j_1$ and $j_2$ plays a role in this calculation. Using similar considerations, in the case of general (fixed) values of $j_1$ and $j_2$ it turns out that we have $j_1 n_1 = j_2 n_2$ whenever
$$
n_1 = v \frac{j_2}{\gcd(j_1,j_2)}, \quad n_2 = v \frac{j_1}{\gcd(v_1,v_2)} \qquad \textrm{for some positive integer $v$}.
$$
As a consequence we have
\begin{eqnarray}
& & \# \Big\{ (n_1, n_2):~1 \leq n_1, n_2 \leq H_L, ~j_1 n_1 = j_2 n_2 \Big\} \nonumber\\
& = & \# \left\{v \geq 1:~ v \leq \min \left( \frac{H_L \gcd(j_1,j_2)}{j_2}, \frac{H_L \gcd(j_1,j_2)}{j_1} \right) \right\} \nonumber\\
& = & \left\lfloor \frac{H_L \gcd(j_1,j_2)}{\max(j_1,j_2)} \right\rfloor \nonumber\\
& \leq & \frac{H_L \gcd(j_1,j_2)}{\sqrt{j_1 j_2}}. \nonumber
\end{eqnarray}
Combining this estimate with~\eqref{rhssum} we obtain
\begin{equation} \label{gcdsum}
\left\| \sum_{h=1}^{H_L} p_L (h \cdot) \right\|_2^2 \leq H_L \sum_{j_1,j_2=1}^{4^{3L}} \frac{|a_{j_1} a_{j_2}|}{2} \frac{\gcd(j_1,j_2)}{\sqrt{j_1 j_2}}.
\end{equation}
The sum on the right-hand side of the last equation is called a \emph{GCD sum}. It is well-known that such sums play an important role in the metric theory of Diophantine approximation; the particular sum in~\eqref{gcdsum} probably appeared for the first time in LeVeque's paper~\cite{LeVe} (see also~\cite{dhs} and~\cite{abs}). A precise upper bound for these sums has been obtained by Hilberdink~\cite{hilber}.\footnote{The upper bounds for the GCD sums in~\cite{hilber} are formulated in terms of the largest eigenvalues of certain GCD matrices; since these matrices are symmetric and positive definite, the largest eigenvalue also gives an upper bound for the GCD sum. This relation is explained in detail in~\cite{absw}.} Hilberdink's result implies that there exists an absolute constant $c_{\textup{abs}}$ such that
$$
\sum_{j_1,j_2=1}^{4^{3L}} \frac{|a_{j_1} a_{j_2}|}{2} \frac{\gcd(j_1,j_2)}{\sqrt{j_1 j_2}} \ll \exp \left( \frac{c_{\textup{abs}} \sqrt{\log (4^{3L})}}{\sqrt{\log \log 4^{3L}}} \right) \sum_{j=1}^{4^{3L}} a_j^2.
$$
Combining this estimate with~\eqref{il2norm} and~\eqref{gcdsum} (and using Parseval's identity) we have
\begin{eqnarray*}
\left\| \sum_{h=1}^{H_L} p_L (h \cdot) \right\|_2^2 & \ll & H_L \exp \left( \frac{c_{\textup{abs}} \sqrt{\log (4^{3L})}}{\sqrt{\log \log 4^{3L}}} \right) \frac{\kappa}{4 Q} \left(\frac{\lambda}{\delta_{i_{L}+1}}\right)^{L} \\
& \ll & (1+\eta)^L \exp \left( \frac{c_{\textup{abs}} \sqrt{\log (4^{3L})}}{\sqrt{\log \log 4^{3L}}} \right),
\end{eqnarray*}
and, together with~\eqref{split} and~\eqref{split2}, and with a similar argument for the odd part of $\mathbb{I}_L$, we obtain
\begin{equation} \label{parts}
\left\| \sum_{h=1}^{H_L} \mathbb{I}_L (h \cdot) \right\|_2^2 \ll (1+\eta)^L \exp \left( \frac{c_{\textup{abs}} \sqrt{\log (4^{3L})}}{\sqrt{\log \log 4^{3L}}} \right).
\end{equation}
By Chebyshev's inequality we have
\begin{eqnarray*}
\p \left( \alpha \in [0,1):~ \left| \sum_{h=1}^{H_L} \mathbb{I}_L (h \alpha) \right| > (\log H_L) \left\| \sum_{h=1}^{H_L} \mathbb{I}_L (h \cdot) \right\|_2 \right) & \leq & \frac{1}{(\log H_L)^2},
\end{eqnarray*}
and since $(H_L)_{L \geq 1}$ grows exponentially in $L$ these probabilities give a convergent series when summing over $L$. Thus by the Borel--Cantelli lemma with probability one only finitely many events
$$
\left| \sum_{h=1}^{H_L} \mathbb{I}_L (h \alpha) \right| > (\log H_L) \left\| \sum_{h=1}^{M} \mathbb{I}_L (h \cdot) \right\|_2
$$
happen, which by~\eqref{parts} implies that 
$$
\left| \sum_{h=1}^{H_L} \mathbb{I}_L (h \alpha) \right| \ll (1 + \eta)^{L/2} \exp \left( \frac{\hat{c}_{\textup{abs}} \sqrt{L}}{\sqrt{\log L}} \right)
$$
for some absolute constant $\hat{c}_{\textup{abs}}$. Comparing this upper bound with~\eqref{summe} and using~\eqref{lefthand} we conclude that
$$
\sum_{h=1}^{H_L}  \mathbf{1}_L (h \alpha) \gg (1+\eta)^L \qquad \textrm{as $L \to \infty$}
$$
for almost all $\alpha$. In particular we have
$$
\sum_{L=1}^\infty \sum_{h=1}^{H_L}  \mathbf{1}_L (h \alpha) = \infty
$$
for almost all $\alpha$, which means that for almost all $\alpha$ infinitely many events~\eqref{F10} occur. As noted after equation~\eqref{F10}, this proves the theorem.
\end{proof}

\section{Concrete Examples: Proof of Theorem~\ref{th_thue4}} \label{sect_proof_th4}

It is known (see~\cite{foumau}, formula~(2.10)) that for all $\alpha$ we have
\begin{equation} \label{F20}
\prod^{L}_{\ell=0} \left|\sin \pi 2^{\ell} \alpha\right| \leq \mathcal{H}^{L}
\end{equation}
for all $L$, where $\mathcal{H} = \frac{\sqrt{3}}{2} = 0.866 \ldots$. Thus from~\eqref{lacproducts},~\eqref{BL3} and the Weyl criterion it follows that $\left(\left\{n_{k} \alpha\right\}\right)_{k\geq 1}$ is u.d. mod 1 iff $\alpha$ is irrational.\\
Hence by~\eqref{erdtur},~\eqref{BL1},~\eqref{lacproducts}, and~\eqref{lacproductsLB} for $N$ of the form $N=2^{L}$ we have
\begin{eqnarray*}
\underset{h\leq N}{\max} \frac{1}{h} \left|\prod^{L}_{l=0} |2 \sin \pi h 2^{l} \alpha | - \prod^{L}_{l=0} |2 \cos \pi h 2^{l} \alpha |\right| & \ll & N\widetilde{D}_{N}^{*} \\
& \ll & \sum^{N}_{h=1} \frac{1}{h\left\|h \alpha\right\|} + \sum^{N}_{h=1} \frac{1}{h} \prod^{L}_{k=0} \left|2 \sin \pi 2^{k} h \alpha\right| \\
& \ll & \sum^{N}_{h=1} \frac{1}{h \left\|h \alpha\right\|}+ (\log N) N^{\frac{\log3}{\log 4}},
\end{eqnarray*}
where we write $\widetilde{D}_{N}^{*}$ for the star-discrepancy of the first $N$ terms of the Thue--Morse--Kronecker sequence.\\

From the left-hand side of this inequality it is not difficult -- but we do not want to go into the details here -- to show that
$$
\underset{n \leq N}{\max} ~n \widetilde{D}_{n}^{*} \gg \frac{1}{\log N} \underset{n \leq N}{\max}~ n D_{n}^{*}
$$
for all $N$, where $D^{*}_{N}$ denotes the star-discrepancy of the pure Kronecker sequence.\\

Furthermore it is easy to show -- we again do not go into the details -- that
$$
\sum^{N}_{h=1} \frac{1}{h \left\|h \alpha\right\|} \ll (\log N) N D_{N}^{*} ~\mbox{for all}~ N=2^{L}.
$$
Hence for all $N$ we have
\begin{equation} \label{BL2}
\frac{1}{\log N} \underset{n \leq N}{\max} ~n D_{n}^{*} \ll \underset{n \leq N}{\max}~ n \widetilde{D}_{n}^{*} \ll \left(\log N\right)^{2} \underset{n \leq N}{\max} ~n D_{n}^{*} + \left(\log N\right)^{2} N^{\frac{\log3}{\log 4}}.
\end{equation}
From this we conclude that the order of the discrepancy $N\widetilde{D}_{N}^{*}$ of the Thue--Morse--Kronecker sequence always is essentially (up to logarithmic factors) larger or equal to the order of the discrepancy $N D_{N}^{*}$ of the pure Kronecker sequence. In fact, the order of $N \widetilde{D}_{N}^{*}$ essentially equals the order of $N D_{N}^{*}$ plus an expression which is at most of order $N^{\frac{\log 3}{\log 4}+ \ve}$. The order of the additional expression is controlled by lacunary products of sine-functions.\\

Hence we conclude what we have already announced in Section~\ref{sect_int}:
\begin{itemize}
\item [-] If the order of $D_{N}^{*}$ satisfies $N D_{N}^{*}= \Omega \left(N^{\frac{\log 3}{\log 4}}\right)$ then $\widetilde{D}_{N}^{*}$ essentially is of the same order as $D_{N}^{*}$
\item [-] If $D_{N}^{*}$ satisfies $N D_{N}^{*} = \mathcal{O} \left(N^{\frac{\log 3}{\log 4}}\right)$ then $\widetilde{D}_{N}^{*}$ satisfies $N \widetilde{D}_{N}^{*}= \mathcal{O} \left(N^{\frac{\log 3}{\log 4} + \ve}\right)$
\end{itemize}
Hence the two examples given in Theorem~\ref{th_thue4} show interesting non-trivial cases where we have (almost) best possible distribution for the pure Kronecker sequence with bad distribution for the Thue--Morse--Kronecker sequence. Indeed especially the first example gives essentially the extremal values for $D_{N}^{*}$ and for $\widetilde{D}_{N}^{*}$, and shows that the right-hand side of~\eqref{BL2} is also essentially optimal.\\

It remains an open problem to give concrete examples $\alpha$ where the corresponding Thue--Morse--Kronecker sequence has ``small'' discrepancy $D_{N}^{*}$, e.g., a discrepancy of the metric order given in Theorem~\ref{th_thue2} or smaller. Of course it also remains an open problem to give good estimates for $\widetilde{D}_{N}^{*}$ in the case of ``natural'' examples of $\alpha$ like $\alpha = \sqrt{2}$.\\

\begin{proof}[Proof of Theorem~\ref{th_thue4}]~\\
\begin{enumerate}
\item [a)] By~\cite{Sh2} we know that the continued fraction coefficients of the number $\beta := \sum^{\infty}_{k=1} \frac{1}{4^{2^{k}}}$ are bounded. This is equivalent with the existence of some $c > 0$ such that $\left|\beta-\frac{p}{q}\right|> \frac{c}{q^{2}}$ for all $p,q \in \mathbb{Z}, q \geq 1,$ i.e.,~~$\left|q \beta -p\right| > \frac{c}{q}$ for all such $p$ and $q$.\\

\noindent This implies especially $\left|3 q \beta+ 2 q-3 p\right| > \frac{c}{3 q}$ for all $p, q \in \mathbb{Z}, q \geq 1$, i.e.,\\
$\left|\beta + \frac{2}{3} - \frac{p}{q}\right| > \frac{\frac{c}{9}}{q^{2}}$ for all such $p,q,$ and hence $ \alpha= \beta + \frac{2}{3}$ has bounded continued fraction coefficients. So the star-discrepancy of the pure Kronecker sequence $\left(\left\{n \alpha\right\}\right)_{n \geq 1}$ satisfies $N D_{N}^{*}= \mathcal{O} \left(\log N\right).$\\

\noindent On the other hand we already know that for the star-discrepancy $\widetilde{D}_{N}^{*}$ of the Thue--Morse--Kronecker sequence $\left(\left\{n_{k} \alpha \right\}\right)_{k \geq 1}$ with $N= 2^{L}$ we have
\begin{eqnarray*}
N \widetilde{D}^{*}_{N} & \gg & \left|\sum^{N}_{k=1} \exp \left(2 \pi i n_{k} \alpha \right)\right| \\
& \gg & \left|\prod^{L}_{\ell=0} \left|2 \sin \pi 2^{\ell} \alpha \right|- \frac{1}{\left\|\alpha\right\|}\right|.
\end{eqnarray*}
We give a suitable lower bound for
\begin{eqnarray*}
\Pi_{L} & := & \prod^{L}_{\ell=0} \left|2 \sin \pi 2^{\ell} \alpha \right| \\
& = & \prod^{L}_{\ell=0} \left|2 \sin \frac{\pi}{3}\right| \prod^{L}_{\ell=0} \frac{\left|\sin \pi 2^{\ell} \alpha \right|}{\left|\sin \frac{\pi}{3}\right|} \\
& = & \sqrt{3} N^{\frac{\log3}{\log 4}} \prod^{L}_{\ell=0} \left|\cos \left(\pi \delta_{\ell}\right)+ \frac{\left(-1\right)^{\ell+1}}{\sqrt{3}} \sin \left(\pi \delta_{\ell}\right)\right|.
\end{eqnarray*}
Here $\delta_{\ell} := \left\{2^{\ell} \beta \right\}$ and we have used $\sin \left(x+y\right) = \sin x \cdot \cos y + \cos x \cdot \sin y.$\\
Note that the base 2 representation of $\left\{2^{\ell} \alpha \right\}$ has one of the following ten possible forms:
$$0.1010\ldots$$
$$0.0101 \ldots$$
$$0.0010 \ldots$$
$$0.1101\ldots$$
$$0.1110\ldots$$
$$0.1001 \ldots$$
$$0.1100 \ldots$$
$$0.0111 \ldots$$
$$0.1011 \ldots$$
$$0.0110 \ldots$$
hence $\left\|2^{\ell} \alpha \right\| > \frac{1}{16}$ always and therefore $\frac{\left|\sin \pi 2^{\ell} \alpha \right|}{\left|\sin \frac{\pi}{3}\right|} > 0.2$ always.\\

\noindent Because of $\left|\cos \pi x -1\right| \leq 3x$ and $\left|\sin \pi x\right| \leq \pi x$ for $x \geq 0$ we have 
$$
\left|\cos \pi \delta_{\ell} + \frac{\left(-1\right)^{\ell+1}}{\sqrt{3}} \sin \pi \delta_{\ell} \right| \geq 1- \left(3+ \frac{\pi}{\sqrt{3}}\right) \delta_{\ell} > 1- 5 \delta_{\ell}.
$$
Therefore, noting that $\max (0.2, ~1-5x) > e^{-11x}$ for $x >0$, we also have
\begin{eqnarray*}
\prod^{L}_{\ell=0} \frac{\left|\sin \pi 2^{\ell} \alpha\right|}{\left|\sin \frac{\pi}{3}\right|} & > & \prod^{L}_{\ell=0} \max \left(0.2, ~1-5 \delta_{\ell}\right) \\
& > & e^{-11 \sum^{L}_{\ell=0} \delta_{\ell}} \\
& \gg & e^{-22 \log L}.
\end{eqnarray*}
So
$$
\Pi_{L} \gg \frac{1}{L^{22}} N^{\frac{\log 3}{\log 4}} \gg \frac{1}{\left(\log N\right)^{22}} N^{\frac{\log3}{\log 4}}
$$
and the lower bound for $\widetilde{D}^{*}_{N}$ follows. The upper bound for $\widetilde{D}^{*}_{N}$ follows from~\eqref{BL2} and from the upper bound for $D^{*}_{N}$.\\

\item [b)] It was shown in~\cite{YB} that $\gamma$ has approximation degree 1, hence for the star-discrepancy of the sequence $\left(\left\{n \alpha \right\}\right)_{n \geq 1}$ we have $N D_{N}^{*} = \mathcal{O} \left(N^{\ve}\right)$ for every $\ve >0.$ To prove the lower bound for the star-discrepancy $\widetilde{D}_{N}^{*}$ of the sequence $\left(\left\{n_{k} \gamma\right\}\right)_{k \geq 1}$, like in the proof of part a) we have to estimate $\Pi_{L} := \prod^{L}_{\ell=0} \left|2 \sin \pi 2^{\ell} \gamma \right|$ from below. We will give in the following as an additional information also an upper estimate for $\Pi_{L}$ in order to show that our lower estimate is rather sharp.\\
We may restrict ourselves to $L$ of the form $L= 8 U -1$. Then
$$
\Pi_{L} = \prod^{U-1}_{j=0} \prod^{8j +7}_{\ell=8 j} \left|2 \sin \pi 2^{\ell} \gamma\right|.
$$
In the following we use some well-known facts on properties of the Thue--Morse sequence:\\
The base 2-representation of $\gamma=0,\gamma_{1} \gamma_{2} \gamma_{3} \ldots$ consists of 8-blocks $\gamma_{8 v+1} \ldots \gamma_{8 v+8}$ of the form $A:= 10010110$ or $B:= 01101001.$\\

\noindent Four such consecutive 8-blocks can occur in the following ten combinations:
$$c_{1} = AABA$$
$$c_{2} = AABB$$
$$c_{3} = ABAA$$
$$c_{4} = ABBA$$
$$c_{5} = ABAB$$
$$c_{6} = BBAB$$
$$c_{7} = BBAA$$
$$c_{8} = BABB$$
$$c_{9} = BAAB$$
$$c_{10} = BABA$$
Let, for example, $j$ be such that $2^{8 j} = 0,c_{1} \ldots.$ \quad ($c_{1}$ is the block of 32 digits defined

\hspace{8,3cm} above)\\
For $m=0,1,2,\ldots, 7$ let
\begin{align*}
x_{1,m} := \left\{ \begin{array}{lll}
2^{m} \cdot 0, c_{1} & \mbox{if} & \left\{2^{m} \cdot 0, c_{1}\right\} < \frac{1}{2}\\
2^{m} \cdot \left(0, c_{1} + \frac{1}{2^{32}}\right) & \mbox{if} & \left\{2^{m} \cdot 0, c_{1} \right\} > \frac{1}{2}\\
\end{array} \right.
\end{align*}
and
\begin{align*}
y_{1,m} := \left\{ \begin{array}{lll}
2^{m} \cdot 0, c_{1} & \mbox{if} & \left\{2^{m} \cdot 0, c_{1}\right\} > \frac{1}{2}\\
2^{m} \cdot \left(0, c_{1} + \frac{1}{2^{32}}\right) & \mbox{if} & \left\{2^{m} \cdot 0, c_{1} \right\} < \frac{1}{2}\\
\end{array} \right.
\end{align*}
\end{enumerate}
Then
$$
\prod^{8 j+7}_{\ell=8 j} \left|2 \sin \pi 2^{\ell} \gamma\right| < \prod^{7}_{m=0} \left|2 \sin \pi y_{1,m}\right|=:U(c_{1}) = 33.487710\ldots 
$$
$$
\prod^{8 j +7}_{\ell=8 j} \left|2 \sin \pi 2^{\ell} \gamma\right| > \prod^{7}_{m=0} \left|2 \sin \pi x_{1,m}\right| =: D(c_{1}) = 33.487705\ldots
$$
In the same way we determine 
$$U\left(c_{i}\right) \mbox{and}~ D\left(c_{i}\right) \quad \mbox{for} ~ i=2,3, \ldots, 10.$$
Further it is well known that the frequencies $F\left(c_{i}\right)$ of the occurrence of a quadruple $c_{i}$ of $8$-blocks in the Thue--Morse sequence are given by 
$$
F\left(c_{1}\right)= F\left(c_{2}\right) = F\left(c_{3}\right)=F\left(c_{5}\right)=F\left(c_{6}\right)=F\left(c_{7}\right)=F\left(c_{8}\right)=F\left(c_{10}\right)= \frac{1}{12}
$$
and
$$
F\left(c_{4}\right)=F\left(c_{9}\right)= \frac{1}{6}.
$$
Hence we get
$$
( 1 - \epsilon )^U \cdot\left(D\left(c_{1}\right) D\left(c_{2}\right) D\left(c_{3}\right) D\left(c_{5}\right) D\left(c_{6}\right) D\left(c_{7}\right) D\left(c_{8}\right) D\left(c_{10}\right)\right)^{\frac{U}{12}} \cdot \left(D\left(c_{4}\right) D\left(c_{9}\right)\right)^{\frac{U}{6}} \leq
$$
$$
\ll \prod^{L}_{\ell=0} \left|2 \sin \pi 2^{\ell} \gamma \right| \ll
$$
$$
(1 + \epsilon )^U \cdot \left(U\left(c_{1}\right) U\left(c_{2}\right) U\left(c_{3}\right) U\left(c_{5}\right) U\left(c_{6}\right) U\left(c_{7}\right) U\left(c_{8}\right) U\left(c_{10}\right)\right)^{\frac{U}{12}} \cdot \left(U\left(c_{4}\right) U\left(c_{9}\right)\right)^{\frac{U}{6}}
$$
which leads to
$$
N^{0.6178775} \ll \prod^{L}_{\ell=0} \left|2 \sin \pi 2^{\ell} \gamma\right| \ll N^{0.6178777}.
$$
for $\epsilon$ small and $L$ large enough. This finishes the proof.
\end{proof}

\section{An open problem from the theory of metric Diophantine approximation} \label{sec_conc}

In conclusion, we mention an open problem from the theory of Diophantine approximation which is related to our proof of the lower bound in Theorem~\ref{th_thue2}. In metric Diophantine approximation, one is often interested in finding conditions on $(\phi(q))_{q \geq 1}$ which guarantee that 
$$
\left|\alpha - \frac{p}{q}\right| < \frac{\phi(q)}{q}
$$
has infinitely many integer solutions $p,q$ for almost all $\alpha$. Two instances of this problem, either under the additional requirement that $p,q$ are coprime (\emph{Duffin--Schaeffer conjecture}) or without this additional requirement (\emph{Catlin conjecture}), constitute probably the two most important open problems in metric number theory. For the origin of the Duffin--Schaeffer conjecture see~\cite{duff}, for the Catlin conjecture see~\cite{catlin}. Problems of this type are discussed in great detail in Glyn Harman's monograph on \emph{Metric Number Theory}~\cite{harman}. For a recent survey, see~\cite{ber}.\\

The problem without the requirement of coprime solutions can also be written in the following form: Let $A_1, A_2, \dots$ be intervals of length $\leq 1$, which are symmetric around 0. Let $\psi_1, \psi_2, \dots$ denote the Lebesgue measure (that is, the length) of these intervals. Under which conditions on $\psi_1, \psi_2, \dots$ do we have
$$
\sum_{n=1}^\infty \mathbf{1}_{A_n} (n \alpha) = \infty
$$
for almost all $\alpha$? Here $\mathbf{1}_A$ denotes the indicator function of $A$, extended with period one.\\

Now in a first step this problem can be generalized to the case when the intervals $A_1, A_2, \dots$ are not necessarily symmetric around $0$, which leads to a problem in \emph{inhomogeneous} Diophantine approximation. This type of question is also quite well-investigated.\\

Perpetuating this line of thought, it is natural to ask what happens if we don't assume that $A_1, A_2, \dots$ are \emph{intervals}, but if they may denote any \emph{measurable sets} in $[0,1]$. Writing $\psi_1, \psi_2, \dots$ for the measure of these sets, the question is under which conditions on $\psi_1, \psi_2, \dots$ we have
$$
\sum_{n=1}^\infty \mathbf{1}_{A_n} (n \alpha) = \infty
$$
for almost all $\alpha$. Note that a necessary condition is the divergence of the sum of the measures, by the Borel-Cantelli lemma. It seems that hardly anything is known about this general problem.  As far as we know, this problem was first stated by LeVeque in~\cite{LeVe}. In this paper he had answered a conjecture of Erd\H{o}s, and he formulated a generalized version of Erd\H{o}s' conjecture. We consider this as a very interesting open problem, and we re-state it below.\\

\textbf{Open problem:} \quad Let $A_1, A_2, \dots$ be measurable sets in $[0,1]$, and let $\psi_1, \psi_2, \dots$ denote their measure. Under which conditions on $(\psi_n)_{n \geq 1}$ is it true that for almost all $\alpha$ the fractional part $\{n \alpha\}$ is contained in the set $A_n$ for infinitely many indices $n$; equivalently, under which conditions is it true that 
$$
\sum_{n=1}^\infty \mathbf{1}_{A_n} (n \alpha) = \infty \qquad \textrm{almost everywhere},
$$
where the indicator functions are extended with period one.\\

A problem quite similar to this one emerged during the proof of the lower bound of Theorem~\ref{th_thue2}. However, the situation was comparatively simple there, for example since we could assume there that the sets $A_n$ can be written as the sum of a moderate number of intervals. The general problem seems to be much more complicated; LeVeque wrote that this general problem ``seems rather intractable''.

\def\cprime{$'$}

\end{document}